\documentclass[12pt,twoside]{amsart}
\usepackage[margin=3cm]{geometry}
\usepackage[colorlinks=false]{hyperref}
\usepackage[english]{babel}
\usepackage{graphicx,titling}
\usepackage{float}
\usepackage{amsmath,amsfonts,amssymb,amsthm}
\usepackage{lipsum}
\usepackage[T1]{fontenc}
\usepackage{fourier}
\usepackage{color}
\usepackage[latin1]{inputenc}
\usepackage{esint}
\usepackage{caption}
\usepackage{ccicons}

\usepackage{enumerate}  
\usepackage[alphabetic, msc-links]{amsrefs}  
\makeatletter
\def\blfootnote{\xdef\@thefnmark{}\@footnotetext}
\makeatother

\newcommand\ccnote{
    \blfootnote{\copyright\,\, Tom Ilmanen and Brian White}
    \blfootnote{\ccLogo\, \ccAttribution\,\, Licensed under a \href{https://creativecommons.org/licenses/by/4.0/}{Creative Commons Attribution License (CC-BY)}.}
}

\usepackage[export]{adjustbox}
\numberwithin{equation}{section}
\usepackage{setspace}
\renewcommand{\le}{\leqslant}

\renewcommand{\ge}{\geqslant}

\renewcommand{\mathbb}{\varmathbb}
\usepackage{fancyhdr}
\newtheorem{theorem}{Theorem}[section]
\newtheorem{lemma}[theorem]{Lemma}
\newtheorem{corollary}[theorem]{Corollary}
\newtheorem{proposition}[theorem]{Proposition}
\newtheorem{definition}[theorem]{Definition}
\newtheorem{remark}[theorem]{Remark}
\address{Tom Ilmanen, Department of Mathematics, E. T. H. Z\"urich,  R\"amistrasse 101, 8092 Z\"urich, Switzerland}
\email{tom.ilmanen@math.ethz.ch}
\medskip
\address{Brian White, Department of Mathematics,  Stanford University,  Stanford, CA 94305, USA} 
\email{bcwhite@stanford.edu}

\begin{document}

\thispagestyle{empty}

\begin{minipage}{0.28\textwidth}
\begin{figure}[H]
\includegraphics[width=2.5cm,height=2.5cm,left]{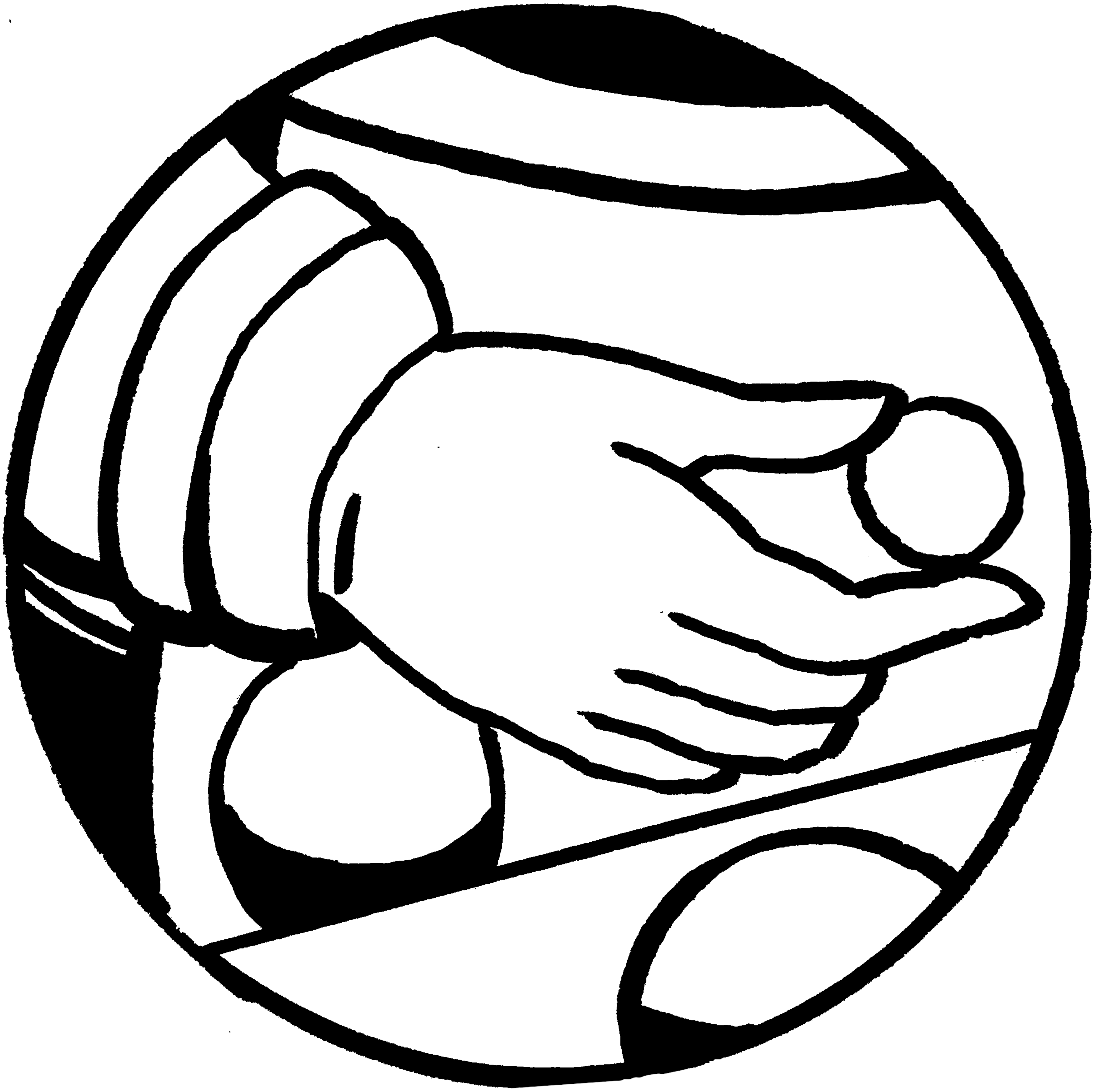}
\end{figure}
\end{minipage}
\begin{minipage}{0.7\textwidth} 
\begin{flushright}
Ars Inveniendi Analytica (2025), Paper No. 4, 32 pp.
\\
DOI 10.15781/cybg-mf32
\\
ISSN: 2769-8505
\end{flushright}
\end{minipage}

\ccnote

\vspace{1cm}


\begin{center}
\begin{huge}
\textit{Fattening in Mean Curvature Flow}


\end{huge}
\end{center}

\vspace{1cm}


\begin{minipage}[t]{.28\textwidth}
\begin{center}
{\large{\bf{Tom Ilmanen}}} \\
\vskip0.15cm
\footnotesize{E. T. H. Z\"urich}
\end{center}
\end{minipage}
\hfill
\noindent
\begin{minipage}[t]{.28\textwidth}
\begin{center}
{\large{\bf{Brian White}}} \\
\vskip0.15cm
\footnotesize{Stanford University}
\end{center}
\end{minipage}

\vspace{1cm}


\begin{center}
\noindent \em{Communicated by Carlo Sinestrari}
\end{center}
\vspace{1cm}


\noindent \textbf{Abstract.} \textit{For each $g\ge 3$, we prove existence of a compact, connected, smoothly embedded, genus-$g$
surface $M_g$ with the following property: under mean curvature flow, 
there is exactly one singular point at the first singular time, and the tangent flow at the singularity is 
given by a shrinker with genus $(g-1)$ and with two ends.  
Furthermore, we show that if $g$ is sufficiently large, then $M_g$ fattens at the first singular time.
As $g\to\infty$, the shrinkers converge to a multiplicity $2$ plane.}
\vskip0.3cm

\noindent \textbf{Keywords.} Mean curvature flow, fattening, level set flow. 
\vspace{0.5cm}


 \newtheorem*{theorem*}{Theorem}
     \newtheorem{claim}{Claim}
 
     \theoremstyle{definition}
    \newtheorem{example}  [theorem]   {Example}

\newcommand{\diag}{\operatorname{diag}}

\newcommand{\Aa}{\mathcal A} 

\newcommand{\Aac}{\mathcal{A}_{\rm con}}

\newcommand{\Aar}{\mathcal{A}_{\rm rect}}

 \newcommand{\area}{\operatorname{area}}

\newcommand{\pdf}[2]{\frac{\partial #1}{\partial #2}}
\newcommand{\pdt}[1]{\frac{\partial #1}{\partial t}}

 \newcommand{\Cc}{\mathcal C}
 
 \newcommand{\diam}{\operatorname{diam}}

    \newcommand{\dist}{\operatorname{dist}}

\newcommand{\ee}{\mathbf e}

 \newcommand{\eps}{\epsilon}

\newcommand{\entropy}{\operatorname{entropy}}
\newcommand{\mcd}{\operatorname{mcd}}
\newcommand{\mdr}{\operatorname{mdr}}
\newcommand{\tM}{\tilde M}
 \newcommand{\Mm}{\mathcal{M}}
   \newcommand{\MM}{\mathcal{M}}
\newcommand{\tF}{\tilde F}
\newcommand{\tmu}{\tilde \mu}

 \newcommand{\genus}{\operatorname{genus}}

\newcommand{\Gin}{\Gamma_\textnormal{in}}

\newcommand{\interior}{\operatorname{interior}}

\newcommand{\Rr}{\mathcal R}
 \newcommand{\RR}{\mathbf{R}}  
 
   \newcommand{\Ss}{\mathbf{S}}

\newcommand{\Tfat}{\operatorname{T_\textnormal{fat}}}
\newcommand{\tTfat}{\operatorname{\tilde T_\textnormal{fat}}}

 \newcommand{\TT}{\mathbf{T}}
 
 \newcommand{\Tan}{\operatorname{Tan}}
 
  \newcommand{\uu}{\mathbf{u}}
 
 \newcommand{\xx}{\mathbf{x}}

  \tableofcontents

\section{Introduction}

In~\cite{white-ICM}, we sketched a proof of the following theorem:

\begin{theorem*}
For every sufficiently large genus $g$, there exists a compact, smoothly embedded,
genus-$g$ surface $M$ in $\RR^3$ 
that fattens (i.e., develops an interior) under the level set mean curvature flow.
\end{theorem*}

That theorem answered a fundamental question that was raised as soon as the level set flow was
introduced~\cite{chen-giga-goto, evans-spruck}.  
Indeed, the last line of~\cite{evans-spruck} is the conjecture
that a compact, smoothly embedded hypersurface in $\RR^n$ never fattens.

In this paper, we use recent work of Bamler and Kleiner
to give a simpler proof of a better result.
In particular, the new theorem gives a more explicit picture of how the fattening occurs.
It produces (and gives descriptions of) interesting singularities 
with shrinkers of every genus $\ge 2$, and it shows that
those singularities cause fattening if the genus is sufficiently large.
Specifically, we prove:

\begin{theorem}\label{intro-theorem}
\begin{enumerate}[\upshape (i)]
\item\label{intro-theorem-part-1} For each integer $g\ge 3$, there exists a compact, connected, smoothly embedded
surface $M_g$ of genus $g$ in $\RR^3$ with the following property: at the first singular time of 
the resulting mean curvature flow,
 there is exactly one singular point, and the tangent flow
at that singularity is given by a shrinker $\Sigma_g$ that has genus $g-1$ and two ends.
(See Figure~\ref{chopp}.)
\item\label{intro-theorem-part-2} As $g\to\infty$, $\Sigma_g$ converges to multiplicity $2$ plane.
\item\label{intro-theorem-part-3} For all sufficiently large $g$, $M_g$ fattens (under level set flow) at the first singular time.
\end{enumerate}
\end{theorem}

\begin{figure}
\begin{center}
\includegraphics[width=.55\textwidth]{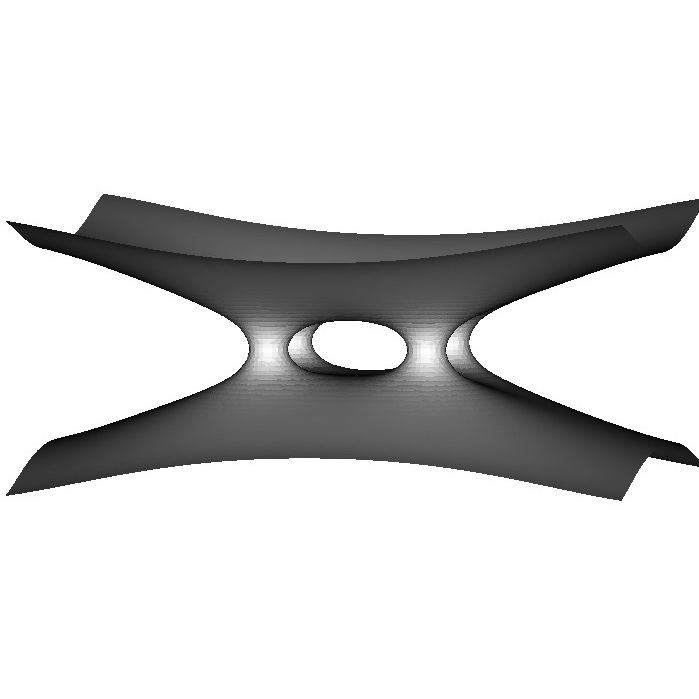}
\caption{\small A shrinker $\Sigma_4$ with genus $3$ and $2$ ends. (Courtesy of David Chopp.)}
\label{chopp}
\end{center}
\end{figure}

Most of the paper (\S\ref{overview-section}--\S\ref{existence-section})
  is devoted to proving~\eqref{intro-theorem-part-1}.
Most of that part of the paper only uses classical mean curvature flow, since it is about behavior
up to the first singular time.
The proof of~\eqref{intro-theorem-part-2} in~\S\ref{behavior-section}
 is much shorter.  
Finally, \eqref{intro-theorem-part-3} follows from~\eqref{intro-theorem-part-2} relatively
 easily by an argument analogous to the ``popping'' argument in~\cite{white-size}.
See~\S\ref{fattening-section}.
 
To prove Theorem~\ref{intro-theorem}, we introduce two classes of surfaces, ``pancakes''
(\S\ref{pancake-section})
and ``$g$-surfaces'' (\S\ref{g-surface-section}),
 and we analyze how such surfaces behave under mean curvature flow.
We hope that those analyses will have other applications.

Theorem~\ref{intro-theorem} asserts that $M_g$ fattens for all sufficiently large $g$.
For such $g$, it follows, by recent work~\cite{chodosh-et-al}
 of Chodosh, Daniels-Holgate, and Schulze, 
that  if $M$ is {\bf any} surface giving rise to a mean curvature  flow in which $\Sigma_g$ appears in a tangent flow, then $M$ fattens.

This paper does not give any bound on how large $g$ needs to be in order for fattening to occur.
But there is numerical evidence that $g=4$ is large enough.
In a prescient 1994 paper~\cite{chopp}, 
 Chopp used numerical methods to
 produce shrinkers that seem to have
the properties of the shrinkers for $g=2$ and $g=4$
that arise in this paper.  He used a continuity method similar
the one in this paper.  
He mentioned that his $g=2$ example might give rise to fattening.
He also noted that his $g=4$ example was closer to a double-density
horizontal plane than his $g=2$ example, and he wrote that (according to ``some preliminary
 investigations of Angenent and Ilmanen'')
  ``this makes it an even likelier candidate for developing interior'' (i.e., for fattening). 
  In another paper~\cite{aci}, Angenent, Chopp, and Ilmanen rigorously 
  proved that if a shrinker is close enough
  (in a certain sense) to a multiplicity-two plane, then the shrinker fattens. 
  They also showed numerically 
   that Chopp's $g=4$ example was ``close enough'' and therefore should fatten.
 
After this paper was written, 
we learned of recent interesting work~\cite{lee-zhao} by 
Tang-Kai Lee and Xinrui Zhao.
Among other things, they prove existence of compact, smoothly embedded
$m$-manifolds in $\RR^{m+1}$ (for $3\le m\le 6$) that fatten under mean curvature flow.
We also learned that Ketover has independently proved (by minimax
methods) existence of shrinkers similar (and perhaps identical) to the ones in this paper.
He proves that the shrinkers give noncompact examples of  fattening
 provided the genus is sufficiently large.
 
 We remark that combining the results of \cite{lee-zhao}, \cite{ketover}, and~\cite{chodosh-et-al}
 gives another proof (in addition to the one in this paper) of 
 existence of {\bf compact}, smoothly embedded surfaces in $\RR^3$ that fatten under mean curvature flow.
 Let $\Sigma$ be one of the fattening shrinkers produced by Ketover.
 By~\cite{lee-zhao}, there is a smoothly embedded closed surface in $\RR^3$ that develops
 a singularity with a tangent flow given by that shrinker.  By~\cite{chodosh-et-al}, the surface then fattens.

\section{Overview}\label{overview-section}

In this section, we describe the intuitions and the logic
 behind the proof of Theorem~\ref{intro-theorem}.
  In particular, we state (without proof) a sequence
of theorems that (we hope) will seem plausible and that 
easily imply Theorem~\ref{intro-theorem}\eqref{intro-theorem-part-1}.
In the subsequent sections, we prove those theorems.

We begin with some notation.
  If $\theta\in \RR$, let $P_\theta$ be the plane spanned by 
  $(\cos\theta, \sin\theta, 0)$ and $(0,0,1)$.
  If $\alpha<\beta$, let $W(\alpha,\beta)$ denote the open wedge
  \[
     W(\alpha,\beta)= \{ (r\cos\theta, r\sin\theta, z): 0<r, \, \alpha< \theta<\beta, \, z\in \RR\}.
  \]
  Thus $W(\alpha,\beta)$ is one of the connected components of the complement of $P(\alpha)\cup P(\beta)$.
  
We let $\Theta:\RR^3\to\RR^3$ be the map
\begin{equation}\label{Theta-equation}
  \Theta(x,y,z) = \left(\frac{d}{d\theta}\right)_{\theta=0} \Rr_\theta(x,y,z) =  (-y,x,0),
\end{equation}
where $\Rr_\theta:\RR^3\to\RR^3$ is rotation by $\theta$ around the $z$-axis, $Z$.

Let $g\ge 2$ be an integer. 
We shall be interested in surfaces $M$ that are symmetric under reflection in $P(0)$
and in $P(\pi/g)$.  Note that such surfaces are invariant under reflection in each
of the planes $P(k\pi/g)$, $k=0,\dots, g-1$,
and also under rotation by $2\pi/g$ about the $z$-axis.
Note also that such a surface $M$ is completely determined by the portion $M\cap \overline{W(0,\pi/g)}$,
since $M$ is obtained from that portion by repeated reflections.

\begin{definition}\label{wheel-definition}
A compact, connected, embedded surface $M$ in $\RR^3$
is called  a {\bf $g$-wheel} (or a {\bf $g$-wagon-wheel}, if you're not
into the whole brevity thing) provided
\begin{enumerate}
\item $M$ is $C^{1,1}$.
\item $M$ is invariant under reflection in the plane $\{z=0\}$.
\item $M$ is invariant under reflection in the planes $P(0)$ and $P(\pi/g)$.
\item\label{wheel-components} $M\cap\{z=0\}$ has exactly $(g+1)$ components, and one of  the components
    lies in $W(-\pi/g,\pi/g)$.
\item\label{wheel-e3} $\nu \cdot \ee_3\ge 0$ on $M\cap \{z>0\}$.
\item\label{wheel-nu}  $\nu\cdot \Theta \le 0$ on $M\cap W(0,\pi/g)$.
\end{enumerate}
The $g$-wheel is called a {\bf strict $g$-wheel}
 if it is smooth and if~\eqref{wheel-e3} and~\eqref{wheel-nu} hold with strict inequality
 at each point of the indicated sets.
\end{definition}

See Figure~\ref{wheels-figure}.

\begin{figure}[htbp]
\begin{center}
\includegraphics[height=3.4cm]{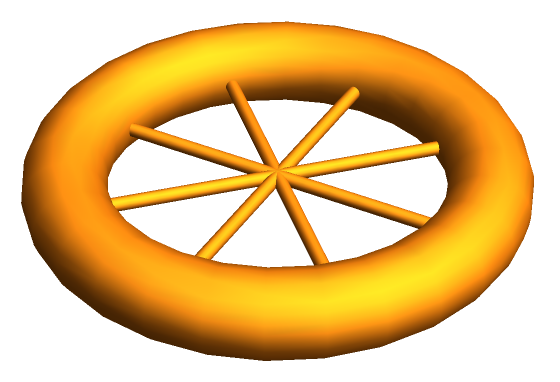}\quad \includegraphics[height=3.5cm]{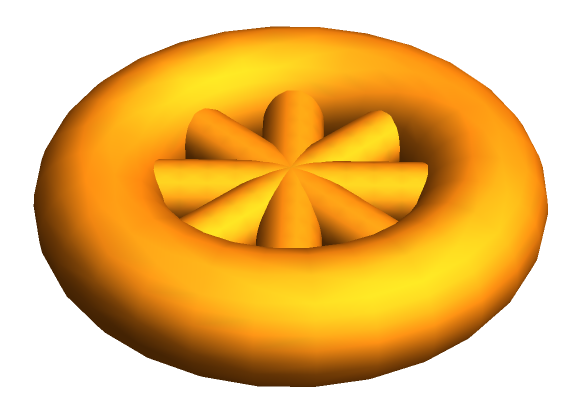}
\caption{Thin and Thick $8$-Wheels}.
\label{wheels-figure}
\end{center}
\end{figure}

If $M$ is a $g$-wheel, then it meets the plane $\{z=0\}$ orthogonally,
and thus each of the components of $M\cap\{z=0\}$ is a $C^{1,1}$ simple closed curve.
The curve is smooth if $M$ is smooth.

\begin{example}\label{example}
Fix an $R>0$ and consider the closed ball $B$ of radius $R$ centered at $(1+R,0,0)$.
Let $\TT=\TT_R$ be the solid torus $\cup_{\theta\in\RR} \Rr_\theta B$.
Thus $\TT\cap\{z=0\}$ is an annulus with inner radius $1$ and outer radius $1+2R$.

Let $K=K_g$ be the union of $\TT$ and the radial segments 
\[
\{s(\cos\theta, \sin\theta,0): 0\le s\le 1\}
\]
such that $\theta$ is an odd multiple of $\pi/g$. (There are $g$ such segments.)

Note that, among points $(r\cos \theta, r\sin\theta, 0)$ with $r\le 1$ and $\theta\in [-\pi/g,\pi/g]$,
there is a unique point $q$ that maximizes its distance to $K$.  Let $\eta=\eta_g$ be that maximum
distance.

For $0<\eps<\eta$, let 
\[
  K(\eps) = \{p: \dist(p,K) \le \eps\}.
\]
Note that $\partial K(\eps)$ is only piecewise smooth, but that it has all the other defining properties
of a $g$-wheel.  In fact, it is easy to smooth $\partial K(\eps)$ to produce a $g$-wheel.

For example, for $0 < s < \eps$, let $M(\eps,s)=M(\eps,s;R)$ be the set of points in $K(\eps)$
at distance $s$ from $\partial K(\eps)$.   Then $M(\eps,s)$ is a $C^{1,1}$ $g$-wheel for all sufficiently small~$s$.
The $C^{1,1}$ smoothness can be shown as follows. 
It is elementary (though somewhat tedious) to show that for all sufficiently small $s>0$, if $p\in K(\eps)$ and $\dist(p, \partial K(\eps))<s$,
then there is unique point in $\partial K(\eps)$ closest to $p$.
 It follows, by~\cite{federer-reach}*{Theorem~4.8(9)},  that $M(\eps,s)$ is $C^{1,1}$ for all sufficiently small $s>0$.

Figure~\ref{wheels-figure} shows $M(\eps,s)$ for two values of $\eps$ in the case $g=8$.
The $s>0$ in Figure~\ref{wheels-figure} is too small to be visible to the naked eye.
\end{example}

\begin{definition}\label{associated-definition}
If $M$ is a $C^{1,1}$ compact, embedded $n$-manifold in $\RR^{n+1}$,
the mean curvature flow {\bf associated} to $M$
is the smooth mean curvature flow $M(t)$, $t\in [0,T)$,
such that $M(0)=M$ and such that the flow becomes singular at time $T$.
The {\bf standard parametrization} of the flow $M(\cdot)$ is the
map $\xx:M\times [0,T)\to \RR^{n+1}$ such that
\begin{equation}\label{classical-mean curvature flow}
\begin{aligned}
&\xx_t = \Delta \xx, \\
&\xx(p,0) =p  \qquad (p\in M)
\end{aligned}
\end{equation}
where $\Delta$ is the Laplacian with respect to the metric $g(t)$ on $M$
such that $\xx(\cdot,t)$ maps $(M,g(t))$ isometrically onto $M(t)$.
\end{definition}

Mean curvature flow preserves symmetries.
Also, standard maximum principle arguments show that the
inequalities in the definition of $g$-wheel immediately become strict inequalities.
  Thus if $M(t)$, $t\in [0,T)$, is the mean curvature flow associated to a
  $g$-wheel $M$,  then $M(t)$ is a strict $g$-wheel
for each $t\in (0,T)$.

It is not hard to show (Theorem~\ref{pancake-regularity-theorem})
that the singularities all occur in the plane $\{z=0\}$.
Thus the singularities form on the curves $M(t)\cap \{z=0\}$. Note that these
curves are smooth for $t\in [0,T)$ since $M(t)$ intersects $\{z=0\}$ orthogonally.
The singular behavior is particularly nice if singularities do {\em not} form on the
outermost component of $M(t) \cap\{z=0\}$:

\newcommand{\Gout}{\Gamma_\textnormal{out}}

\begin{definition}\label{Gout-definition}
If $M$ is a $g$-wheel, we let $\Gout(M)$ denote the outermost
component of $M\cap\{z=0\}$, i.e., the boundary of the unbounded
component of $\{z=0\}\setminus M$.
Let $M(t)$, $t\in [0,T)$, be the mean curvature
flow associated to the $g$-wheel $M$.
Then we sometimes write $\Gout(t)$ for $\Gout(M(t))$.
Also, we define
\[
  \Gout(T) :=  \{q: \liminf_{t\to T} \dist(q,\Gout(t)) = 0 \}.
\]
 We say that $M$ is {\bf special}
if all the points $(p,T)$ with $p\in \Gout(T)$ are regular points of
the flow $M(\cdot)$.
\end{definition}

Note that if $M$ is special, then $\Gout(T)$ is a smooth, simple closed
curve.

\begin{theorem}[Trichotomy Theorem, Abbreviated Version]\label{trichotomy-theorem}
Suppose that $g\ge 3$ and that $M$ is a special $g$-wheel.  Let $M(t)$, $t\in [0,T)$,
  be the associated mean curvature
flow. 
Then exactly one of the following holds:
\begin{enumerate}[\upshape(1)]
\item\label{tri-thin} All tangent flows at time $T$ singularities 
  are shrinking cylinders with horizontal axes.
\item\label{tri-thick} All tangent flows at time $T$ singularities
   are shrinking cylinders with vertical axis.
     In this case, there are exactly $g$ points $q$
   such that $(q,T)$ is a singular point of the flow.
\item\label{tri-critical} The spacetime point $(0,T)$ is the only singularity of the flow, and if
    \[
       t\in (-\infty, 0] \mapsto |t|^{1/2} \Sigma
    \]
    is a tangent flow at $(0,T)$, 
     then $\Sigma$ has genus $g-1$ and two ends, one in $\{z>0\}$ and the other in 
    $\{z<0\}$.  Furthermore, $\Sigma^+:=\Sigma\cap\{z>0\}$ projects diffeomorphically
    onto the unbounded component of $\{z=0\}\setminus\Sigma$.
\end{enumerate}
\end{theorem}

(The tangent flow in case~\eqref{tri-critical} is actually unique.
See Corollary~\ref{unique-corollary}.)

If we think of $M$ as a wagon wheel, then, intuitively,  \eqref{tri-thin} should happen
if the spokes of $M$ are very thin, and~\eqref{tri-thick} should  happen if the spokes are very thick.
That suggests the following definition:

\begin{definition}\label{tri-definition}
If $M$ is a special $g$-wheel, we say that $M$ is {\bf thin}, {\bf thick}, or {\bf critical} 
according to whether~\eqref{tri-thin}, \eqref{tri-thick}, or~\eqref{tri-critical} holds in
 Theorem~\ref{trichotomy-theorem}.
\end{definition}

Let $M(\eps,s)$ be as in Example~\ref{example} above.
If $\eps>0$ is sufficiently small, then $M(\eps,s)$ is thin.
It $\eps$ is sufficiently close to $\eta$ and if $s$ is sufficiently small, then $M(\eps,s)$
is thick.  
(These assertions follow from Theorems~\ref{thin-theorem} and~\ref{thick-theorem}.)

We shall show (Theorem~\ref{open-theorem})
  that the sets of thin $g$-wheels  and of thick $g$-wheels are  open subsets of the set
of all special $g$-wheels.   Now if $R$ is sufficiently large, then all the $M(\eps,s)$
in Example~\ref{example} are special (Lemma~\ref{special-lemma}).   Thus 
\begin{equation}\label{the-family}
    \{ M(\eps,s): \eps\in (0,\eta), \, s\in (0,\eps)\}
\end{equation}
is a set of special $g$-wheels that contains both thick and thin examples.
Since the family~\eqref{the-family} is connected, it must contain a critical $g$-wheel,
and hence Theorem~\ref{intro-theorem}\eqref{intro-theorem-part-1} holds.

To prove Theorem~\ref{intro-theorem}\eqref{intro-theorem-part-2}, i.e., that the $\Sigma_g$ converge as $g\to\infty$
to a multiplicity~$2$ plane, we first get curvature estimates and therefore smooth subsequential convergence
away from an annulus in the plane $\{z=0\}$.  Of course, the symmetries of the $\Sigma_g$ imply that limit must be rotationally
invariant about the $Z$-axis.  We then show that the annulus where the convergence is not smooth is actually just a circle,
and that the limit of the $\Sigma_g$ is the union of two smooth, rotationally invariant shrinkers, each of which contains that circle.
Finally, we show that each of those two shrinkers is the plane $\{z=0\}$.

The intuition for Theorem~\ref{intro-theorem}\eqref{intro-theorem-part-3}, i.e., that a critical $g$-wheel fattens if $g$ is large,
is as follows.
we show that, at the first singular time $T$, the surface $M(T)$ is singular at the origin, and the tangent cone $C$ to $M(T)$ at the origin
is the union of two graphs $z=f(x,y)$ and $z= -f(x,y)$, where $f\ge 0$, with strict inequality except at the origin.
It follows that $M(T)$ is, near the origin, the union of two graphs that touch only at the origin.  One way that the flow can extend past time $T$
is for those two graphs to detach from each other.  If $g$ is large, then (by Theorem~\ref{intro-theorem}\eqref{intro-theorem-part-2}),
that tangent cone $C$ is nearly flat.  It follows that there is another way for the flow to extend past time $T$: in a small ball around the origin, the union
of the two graphs can get replaced by an annulus.  Such non-uniqueness is equivalent to fattening.

An earlier version of this paper also handled the case $g=2$.
In particular, Theorems~\ref{intro-theorem} and~\ref{trichotomy-theorem}
remain true for $g=2$.  
But that case is somewhat special: see Remark~\ref{g2-remark}.
Assuming $g\ge 3$ simplifies the proofs, so we
have chosen to make that assumption.

\section{The Bamler-Kleiner Theory}\label{bamler-kleiner-section}

Let $M(t)$, $t \in [0,T)$, be a mean curvature flow of compact, smooth, embedded
surfaces in $\RR^3$.
Let $\Mm$ be the class of flows obtained from $M(\cdot)$ by any combination
of spatial rigid motions, spacetime translations, and parabolic dilations.  
Let $\Cc$ be the closure of $\Mm$ under
 taking weak limits (e.g, in the sense of Brakke flows).

\begin{theorem}\label{bk-theorem}
If $M'(t)$, $t\in I$, is in the class $\Cc$, then $M'(t)$ is smooth and has multiplicity one
for almost every $t\in I$.  
\end{theorem}

In particular, if a self-similar flow
\[
   t\in (-\infty,0] \mapsto |t|^{1/2}\Sigma
\]
is the $t<0$ portion of a flow  in the class $\Cc$, then $\Sigma$ is smooth and has multiplicity one.
Hence for every tangent flow to $M(\cdot)$ (or, more generally, for every tangent flow to any
  flow in $\Cc$), the corresponding shrinker is smooth and has multiplicity $1$.

Theorem~\ref{bk-theorem}
  is a special case of the recent work~\cite{bamler-kleiner} of Bamler and Kleiner.

\section{Pancakes}\label{pancake-section}

A {\bf pancake} in $\RR^3$ is a $C^{1,1}$ embedded, compact, connected surface
that is symmetric about the plane $\{z=0\}$ and that satisfies the inequality
\[
  \text{$\nu\cdot\ee_3\ge 0$ on $M^+:=M\cap\{z>0\}$}.
\]
If the pancake is smooth and if $\nu\cdot \ee_3>0$ at all points of $M^+$, 
we say that $M$ is a {\bf strict pancake}.

\begin{theorem}\label{instant-pancake-theorem}
Suppose that $M$ is a pancake and that $M(t)$, $t\in [0,T)$, 
is the associated mean curvature flow (see Definition~\ref{associated-definition}.)
Then $M(t)$ is a strict pancake for all $t\in (0,T)$.
\end{theorem}

\begin{proof}
Let $\xx: M\times [0,T)\to \RR^3$ be the standard parametrization
of the flow $M(\cdot)$.  (See Definition~\ref{associated-definition}.)

Let $M^+=M\cap\{z>0\}$ and $M(t)^+=M(t)\cap \{z>0\}$.
Since mean curvature flow preserves symmetries, $\xx(\cdot,t)$ maps
$\overline{M^+}$ diffeomorphically on $\overline{M(t)^+}$ for each $t$.

Let $\nu(p,t)$ denote the unit normal to $M(t)$ at $\xx(p,t)$.
Recall that $\nu$
satisfies the evolution equation
\[
  \partial_t \nu - \Delta \nu = |A|^2\nu,
\]
where $|A(p,t)|$ is the norm of the second fundamental form of $M(t)$ at $\xx(p,t)$.
Thus $u:=\nu\cdot \ee_3$ satisfies
\[
  u_t - \Delta u = |A|^2 u.
\]
Note that $u(\cdot,0)=1$ at the highest point on $M^+$
and that $u\equiv 0$ on $\partial(M^+) \times [0,T)$.
Also, $u\ge 0$ on $M^+\times \{0\}$ since $M$ is a pancake.
Consequently, by the strong maximum principle, 
 $u(p,t)>0$ for all $(p,t)$ with $p\in M^+$ and $t\in (0,T)$. 
\end{proof}

The proof of Theorem~\ref{instant-pancake-theorem} also gives the following:

\begin{lemma}\label{local-pancake-lemma}
Suppose $t\in [a,b] \mapsto M(t)$ is a smooth mean curvature flow of surfaces properly embedded
in an open subset $U$ of $\RR^3$.
Suppose that $U$ is an open subset of $\RR^3$
and that 
\[
   \nu\cdot \ee_3\ge 0
\]
on $M(t)$ for all $t\in [a,b]$.
If $\nu\cdot \ee_3=0$ at $p\in M(b)\cap U$, 
then there is an $\eps>0$ such that
\[
   \nu\cdot\ee_3\equiv 0
\]
on $(q,t)$ for $t\in [b,b-\eps]$ and $q\in M(t)\cap B(p,\eps)$.
\end{lemma}

\begin{proof}
Let $W$ be a neighborhood of $p$ in $M(b)$ such that $\overline{W}$ is a compact
subset of $U$.  By ODE, there is an $\eps>0$ and a 
smooth map
\[
  \xx: W\times [b-\eps,b] \to \RR^3
\]
such that 
\begin{gather*}
\xx(q,b)=q, \\
\xx(q,t) \in M(t), \, \text{and}\\
\partial_t \xx(q,t) \in \Tan(M(t), \xx(q,t))^\perp
\end{gather*}
for all $t\in [b-\eps,b]$ and $q\in W$.
It follows that 
\[
 \partial_t \xx = \Delta\xx.
\]
Thus $u(q,t):= \nu(M(t),\xx(q,t))\cdot \ee_3$ satisfies
\[
   (\partial_t - \Delta) u = |A|^2 u.
\]
The conclusion now follows immediately from the strong maximum principle.
\end{proof}

\begin{theorem}\label{graphical}
Suppose that $M$ is a strict pancake.
Then the projection 
 $\pi(x,y,z) = (x,y)$
maps $M^+:=M\cap \{z>0\}$ diffeomorphically onto its image.
\end{theorem}

\begin{proof}
Let $K$ be the compact region bounded by $M$.
By hypothesis, $\pi|M^+$ is a smooth immersion, so it suffices
to show that it is one-to-one.  Suppose to the contrary that there
is a vertical line $L$ that intersects $M^+$ in two or more points.
Let $p$ be the highest point  in $M^+\cap L$, and let $q$ be the next highest point.
It follows that the closed interval $pq$ belongs to $K$, and thus that $\nu(p)\cdot\ee_3>0$
and $\nu(q)\cdot \ee_3<  0$, a contradiction.
\end{proof}

\begin{remark}\label{genus-remark}
Recall that the genus of an orientable surface $M$ 
   without boundary  is the maximum number 
of disjoint, simple closed curves in $M$ that do not disconnect
any connected component of $M$.  
It follows that  if $U$ is an open subset of $M$,
 then $\genus(U)\le \genus(M)$, and that 
if $U_1\subset U_2\subset \dots$  is an exhaustion of $M$ by open subsets,
then $\genus(M)=\lim \genus(U_i)$.
Recall also that the genus and the Euler characteristic $\chi$ are
related by
\[
   \genus(M) = c(M) - \frac12 \chi(M) - \frac12 e(M),
\]
where $c(M)$ is the number of connected components of $M$
and $e(M)$ is  the number
of ends of $M$. 
\end{remark}

\begin{theorem}\label{pancake-genus-theorem}
Suppose that $M$ is a strict pancake. 
Then 
\[
 \genus(M) = a - 1,
\]
where $a$ is the number of components of $M\cap\{z=0\}$.
\end{theorem}

\begin{proof}
Let $M^+=M\cap \{z: z>0\}$ and $M^- = M\cap\{z<0\}$.
Let  $W=\pi(M^+)=\pi(M^-)$.
Then $W$ is a topologically a disk with $a-1$ disks removed from it.
Thus 
\[
  \chi(M^+)=\chi(M^-)= \chi(W)=1- (a-1)= 2-a.
\]
Removing a simple closed curve from a surface does not change its Euler characteristic,
so
\[
 \chi(M) = \chi(M^+\cup M^-) =  4 - 2a.
\]
Thus, since $M$ is connected,
\[
  \genus(M)= 1 -\frac12\chi(M) = a - 1.
\]
\end{proof}

\begin{remark}\label{gout-remark}
The proof of Theorem~\ref{pancake-genus-theorem}
shows that $M\cap\{z=0\}$ has one outer component $\Gout(M)$, which is the boundary
of the unbounded component of $\{z=0\}\setminus M$; the other ``inner'' components
  of $M\cap\{z=0\}$
are contained in the disk whose boundary is $\Gout(M)$.   
Theorem~\ref{pancake-genus-theorem}
 states that the genus of $M$ is equal to the number of inner components
of $M\cap\{z=0\}$.
\end{remark}

\begin{theorem}\label{pancake-regularity-theorem}
Suppose that $M$ is a pancake,  and let $M(t)$, $t\in [0,T)$, 
be the associated mean curvature flow. 
Suppose that $p$ is a point with $z(p)\ne 0$.
Then $(p,T)$ is a regular point of the flow.
\end{theorem}

\begin{proof}
Let $M'(t)$ be the image of $M(t)$ under reflection in the plane $P:=\{z=z(p)\}$.
By Theorems~\ref{instant-pancake-theorem} and~\ref{graphical}, 
\[
  M(t)\cap M'(t) \subset P
\]
for all $t\in (0,T)$.

By a version of the 
Hopf Boundary Point Principle~\cite{choi-h-h-w}*{Theorem~3.19}, $(p,T)$ is regular.
(One applies Theorem~3.19 of~\cite{choi-h-h-w} to the flows $M(\cdot)$ and $M'(\cdot)$. 
That theorem requires that $(p,T)$ be a ``tame point" of the flow.
The Bamler-Kleiner theory (\S\ref{bamler-kleiner-section}) asserts that for mean curvature
flows
associated to smooth,  
closed, embedded surfaces in $\RR^3$, all shrinkers that arise are smooth with multiplicity one.
In particular, all spacetime points are tame for such flows.)
\end{proof}

\begin{theorem}\label{pancake-shrinker-theorem}
Suppose that $M$ is a pancake and let $M(t)$, $t\in [0,T)$, be the associated mean
curvature flow.
Suppose that $(p,T)$ is a singular point of the flow, and that
\[
  t\mapsto |t|^{1/2}\Sigma
\]
is a tangent flow at $p$.  If $\Sigma$ has genus $0$,
it is one of the following: a sphere, a horizontal cylinder, or a vertical cylinder.
If $\Sigma$ has genus $>0$, then:
\begin{enumerate}[\upshape(1)]
\item\label{p-shrinker-graph}
 $\Sigma^+:=\Sigma\cap\{z>0\}$
is the graph of a smooth function $f: \pi(\Sigma^+)\to (0,\infty)$.
\item\label{p-shrinker-cone} As $t\to 0$, $|t|^{1/2}\Sigma$ 
converges to a cone $C=C(\Sigma)$. The convergence is smooth with multiplicity one
in $\{z\ne 0\}$. Furthermore, $\nu\cdot\ee_3>0$ at each point of  $C^+:= C\cap\{z>0\}$.
\item\label{p-shrinker-bound} There is a constant $\lambda<\infty$ such that
such that
\[
   f(x,y) \le \lambda(|(x,y)| + 1)
\]
for all $(x,y)$ in the domain of $f$.
\item\label{p-shrinker-cone-2}
Suppose that $\Sigma$ is noncompact and  that $\Sigma\cap\{z=0\}$ is compact,
and let $C$ be as in Assertion~\eqref{p-shrinker-cone}. 
   Then $C\cap\{z=0\} = \{0\}$, the convergence of $|t|^{1/2}\Sigma$ to $C$ is smooth
    except at the origin, and $C^+\setminus \{0\}$ is the graph of a smooth
    function on $\RR^2\setminus\{0\}$.
\end{enumerate}
\end{theorem}

\begin{proof}
By the Bamler-Kleiner theory (\S\ref{bamler-kleiner-section}), $\Sigma$ is smooth with
multiplicity one.

First, suppose that $\Sigma$ has genus $0$.
By~\cite{brendle}, it is a plane, a sphere, or a cylinder.
If $\Sigma$ were a plane, it would be a multiplicity $1$ plane, and hence $(p,T)$ would be a regular point,
contrary to the hypothesis.  
Thus $\Sigma$ is a sphere or cylinder.
By Theorem~\ref{pancake-regularity-theorem}, $p$ is contained in the plane $\{z=0\}$.  
Consequently,  $\Sigma$ is invariant under $(x,y,z)\mapsto (x,y,-z)$.
Thus, if $\Sigma$ is a cylinder, then the cylinder is either horizontal or vertical.

Now suppose that $\genus(\Sigma)>0$.
Consider the mean curvature flow
\begin{equation}\label{sigma-plus-flow}
  t\in (-\infty,0)\mapsto |t|^{1/2}\Sigma \tag{*}
\end{equation}
Note that $\nu\cdot \ee_3\ge 0$ at all points of $\Sigma^+$.
If $\nu\cdot\ee_3$ were equal to $0$ at some point $p$ of $\Sigma^+$, then (by the strong maximum principle)
it would be $0$ on some neighborhood  of $\Sigma$ (see Lemma~\ref{local-pancake-lemma}).
By unique continuation, it would be $0$ everywhere on $\Sigma$,
  and thus $\Sigma$ would have the form $S\times \RR$
for an embedded shrinker curve $S\subset\RR^2$.  But  then $\Sigma$ would have
genus $0$, contrary to our assumption.   Thus 
  Assertion~\eqref{p-shrinker-graph} holds.

Convergence of $|t|^{1/2}\Sigma$ to a cone $C$ (as $t\to 0$) holds for any shrinker.
Exactly the same argument used to prove Theorem~\ref{pancake-regularity-theorem}
 shows that the flow~\eqref{sigma-plus-flow} has no time $0$ singularities 
 in $\{z>0\}$ (or, by symmetry, in $\{z<0\}$).
  Hence $C^+$ is smooth with multiplicity $1$.  Applying the maximum
 principle (Lemma~\ref{local-pancake-lemma}, with $U=\{z>0\}$) to the flow
 \begin{align*}
 t\in (-\infty,0] 
 \mapsto
 \begin{cases}
 |t|^{1/2}\Sigma   &\text{if $t<0$, and} \\
 C                        &\text{if $t=0$},
 \end{cases}
 \end{align*}
 we see that $\nu\cdot \ee_3>0$ at all points of $C^+$.  
 Thus Assertion~\eqref{p-shrinker-cone} holds.

If the cone $C$ contained any point of $Z^+:=\{(0,0,z): z>0\}$,
 then it would contain all the points of $Z^+$.
But then $\nu\cdot \ee_3$ would be $0$ at those points, contrary to Assertion~\eqref{p-shrinker-cone}.
Thus
\begin{equation}\label{no-Z}
   C\cap Z^+ =\emptyset.
\end{equation}
If the bound on $f$ in Assertion~\eqref{p-shrinker-bound} failed,
 there would be points $p_i = (x_i,y_i,f(x_i,y_i))$ such
that
\[
    \frac{f(x_i,y_i)}{1 + |(x_i,y_i)|} \to \infty.
\]
Thus 
\[
    |p_i| =  |(x_i,y_i,f(x_i,y_i))| \to \infty
\]
and
\[
   \frac{p_i}{|p_i|} \to \ee_3.
\]
Hence, $\ee_3$ would be in $\Sigma^+(0)$ (since $p_i/|p_i|$ is in $(1/|p_i|)\Sigma$).
But that contradicts~\eqref{no-Z}.  Thus Assertion~\eqref{p-shrinker-bound} holds.

Now suppose that $\Sigma$ is noncompact and that $\Gamma:=\Sigma\cap\{z=0\}$ is compact.
If $\Sigma$ had an asymptotically cylindrical end, then (by symmetry) $\Sigma^+$ would have such an end.
But that is impossible by Assertion~\eqref{p-shrinker-bound}.  Thus $\Sigma$ has no asymptotically cylindrical ends.
It follows, by~\cite{bamler-kleiner}*{Theorem~1.10}, that $C\setminus\{0\}$ is smoothly embedded and that the convergence of $|t|^{1/2}\Sigma$ to $C$
is smooth and with multiplicity one away from the origin.  If $C\cap\{z=0\}$ contained a point $p$, then $\nu(C,p)\cdot\ee_3=0$
by $(x,y,z)\mapsto (x,y,-z)$ symmetry of $\Sigma$.   But that is impossible since (for each $\eps>0$), $(|t|^{1/2}C)\setminus B(0,\eps)$
is disjoint from $\{z=0\}$ for all $t<0$ sufficiently close to $0$.  
Hence
\begin{equation}\label{cone-point}
C\cap\{z=0\}  = \{0\}.
\end{equation}
Finally, $C^+\setminus\{0\}$ is the graph of a smooth function over $\RR^2\setminus\{0\}$ 
  by Assertions~\eqref{p-shrinker-graph} and~\eqref{p-shrinker-cone} together with~\eqref{cone-point}.
\end{proof}

\begin{corollary}\label{pancake-shrinker-genus-corollary}
If the shrinker $\Sigma$ in Theorem~\ref{pancake-shrinker-theorem} is noncompact
and if $\Sigma\cap\{z=0\}$ is compact, then the number
of components of $\Sigma\cap\{z=0\}$ is $\genus(\Sigma)+1$.
\end{corollary}

\begin{proof}
If $\Sigma$ has genus $0$, then $\Sigma$ is a vertical cylinder,
so the assertion holds in that case.  

Now suppose that $\genus(\Sigma)>0$ and that $\Sigma\cap\{z=0\}$
has $a$ components.  Each component is a simple closed curve.
Note that $\Sigma^+$ is homeomorphic to the plane with $a$ disks removed,
and hence $\chi(\Sigma^+)=1-a$.  Thus (since removing simple closed curves
does not change the Euler characteristic)
\begin{align*}
\chi(\Sigma)
&=
\chi(\Sigma \setminus\{z=0\})  \\
&= \chi(\Sigma^+) + \chi(\Sigma^-) \\
&=2\chi(\Sigma^+) \\
&= 2-2a.
\end{align*}
Since $\Sigma$ is connected and has $2$ ends, it follows that $\Sigma$
has genus $a-1$.
\end{proof}

\begin{corollary}\label{distance-corollary}
If $M$ is a pancake and if $(p,T)$ is a singular point of the associated mean curvature
flow $M(t)$, $t\in [0,T)$, then
\[
    \limsup_{t\to T} \frac{\dist(p, M(t)\cap\{z=0\})}{|T-t|^{1/2}} < \infty,
\]
and therefore
\[
   \lim_{t\to T} \dist(p, M(t)\cap \{z=0\}) = 0.
\]
\end{corollary}

\begin{proof}
Let $\alpha$ be the $\limsup$ in the assertion.  Choose $t_i\to T$ such that
\[
   \frac{\dist(p, M(t_i)\cap\{z=0\})}{|T-t_i|^{1/2}} \to \alpha.
\]
Then, after passing to a subsequence, 
\[
   |T-t_i|^{-1/2} (M(t_i) - p)
\]
converges smoothly and with multiplicity $1$ to a shrinker $\Sigma$ as described in 
 Theorem~\ref{pancake-regularity-theorem}.
Since $\Sigma$ intersects $\{z=0\}$ orthogonally in a nonempty curve or collection
of curves, it follows that
\[  
   \alpha = \dist(0,\Sigma\cap\{z=0\}) < \infty.
\]
\end{proof}

The following corollary is not needed in this paper.
We include it for completeness.

\begin{corollary}\label{unique-corollary}
Let $\Sigma$ be as in Theorem~\ref{pancake-shrinker-theorem}.
If $\Sigma$ has genus $0$, or if $\genus(\Sigma)>0$ and $\Sigma\cap\{z=0\}$
is compact, then 
\[
 t\in (-\infty,0)\mapsto |t|^{1/2}\Sigma
\]
is the unique tangent flow to $M(\cdot)$ at the spacetime point $(p,T)$.
\end{corollary}

\begin{proof}
If $\Sigma$ is a sphere, uniqueness follows from Huisken's Theorem~\cite{huisken}.
If $\Sigma$ is a cylinder, uniqueness was proved by Colding and Minicozzi~\cite{colding-minicozzi}.
Now suppose that $\genus(\Sigma)>0$. 
If $\Sigma$ is compact, uniqueness was proved by~Schulze~\cite{schulze}.
If $\Sigma$ is noncompact and $\Sigma\cap\{z=0\}$ is compact,
then $\Sigma$ is smoothly asymptotic at $\infty$ to a cone $C$
 (by Assertion~\eqref{p-shrinker-cone-2} of Theorem~\ref{pancake-shrinker-theorem}),
and thus uniqueness holds according to a Theorem of~Chodosh and~Schulze~\cite{chodosh-schulze}.
\end{proof}

Let $\Sigma$ be a smooth surface in $\RR^3$.  Recall that $\Sigma$ is a shrinker if and only if
it is a minimal surface with respect to the shrinker metric, i.e., the metric
\begin{equation}\label{shrinker-metric}
   (4\pi)^{-1}\exp(-|p|^2/4)\, \delta_{ij}.
\end{equation}

\begin{theorem}\label{stable-theorem}
Suppose that $\Sigma$ is an embedded shrinker in $\RR^3$ and that $\Sigma$ is
not the plane $\{z=0\}$.  Let $\Omega$ be
one of the components of $\{z=0\}\setminus \Sigma$.  
Then $\Omega$ is a stable minimal surface with respect to the shrinker metric.
(That is, the second variation of shrinker area is nonnegative for all variations given by compactly supported
variations.)
\end{theorem}

\begin{proof}
Suppose, to the contrary, that $\Omega$ is unstable as a minimal surface with respect to the shrinker metric.
Then there exists an unstable subregion $\Omega'$ of $\Omega$ such that $\overline{\Omega'}$ is a compact
region in $\Omega$ with smooth boundary. Let $W$ be the connected component of $\RR^3\setminus \Sigma$
that contains $\Omega'$.  
Let $M$ be a surface  that minimizes area among all surfaces (i.e., integer-multiplicity currents)
in $\overline{W}$ with boundary $\partial \Omega'$.
(Such a minimizer exists by Lemma~\ref{standard-lemma} below.)
By the strong maximum principle, the support of $M$ does not touch $\Sigma$. 
(See the theorem on page~686 and Remark~(2) on page~691 of \cite{solomon-white}.)
That is, $M$ lies in $W$.
By~\cite{hardt-simon}, $M$ is a smoothly embedded manifold-with-boundary.
Since $\Omega'$ is unstable, $M\ne \Omega'$. Thus $M$ does not lie in $\{z=0\}$, so
 $M^+:=M\cap\{z>0\}$ or $M^-:=M\cap\{z<0\}$ is nonempty.
We may suppose that $M^+$ is nonempty.  Now $M^+$ is stable.  
Thus, by~\cite{brendle}*{Proposition~5}, it is a halfplane bounded by a line through the origin.
Now $\Sigma$ is disjoint from $M$ and therefore from $\overline{M^+}$.
But no shrinker is disjoint from such a closed halfplane by~\cite{white-mcf-boundary}*{Theorem~15.1}.
\end{proof}.

\begin{lemma}\label{standard-lemma}
Let $K$ be a closed subset of $\RR^3$.
Let $A$ be a $2$-dimensional integer-multiplicity current supported in $K$
such that the shrinker area of $A$ is finite.  Then there exists 
a current $M$ that minimizers shrinker-area among all $2$-dimensional integer-multiplicity currents in $K$ with boundary $\partial A$.
\end{lemma}

The proof is standard, but we include it for the reader's convenience.

\begin{proof}
Let $\alpha$ be the infimum of the shrinker area of such currents.  Let $M_i$ be a sequence of such currents whose
shrinker areas tend to $\alpha$.
By the Federer-Fleming Compactness Theorem~\cite{simon-gmt}*{Theorem~27.3},
   the $M_i-A$ converge weakly (after passing to a subsequence)
 to a limit.  Thus the $M_i$ converge weakly to a limit $M$.

Consider the set $\Omega$ of smooth, compactly supported $2$-forms $\omega$ on $\RR^3$ such that
\[
   |\omega(p)| \le (4\pi)^{-1}\exp(-|p|^2/4)
\]
at each point $p$.     For $\omega\in \Omega$, we have
\[
   (M,\omega) = \lim (M_i, \omega).
\]
For each $i$, $(M_i,\omega)\le \area_s(M_i)$, where $\area_s$ is area (i.e., mass) with respect
to the shrinker metric.  Thus
\[
  (M, \omega) \le \liminf \area_s(M_i)) = \alpha.
\]
Consequently, 
\[
\alpha \le  \area_s(M) = \sup_{\omega\in \Omega}(M,\omega) \le \alpha.
\]
Hence $M$ attains the minimum $\alpha$.
\end{proof}

\section{$g$-Surfaces}\label{g-surface-section}

Recall that $\Rr_\theta:\RR^3\to \RR^3$ is rotation by $\theta$ around the $z$-axis,
that $\Theta:\RR^3\to \RR^3$ is the vectorfield
\begin{equation}\label{Theta-recalled}
  \Theta(x,y,z) = \left(\frac{d}{d\theta}\right)_{\theta=0} \Rr_\theta(x,y,z) =  (-y,x,0),
\end{equation}
and that $W(\alpha,\beta)$ is the open wedge in $\RR^3$ given by
\[
 W(\alpha,\beta) = \{ (r\cos\theta, r\sin\theta, z): r>0, \, \alpha< \theta< \beta\}.
\]

\begin{definition}\label{g-surface-definition}
Let $g\ge 2$ be an integer.
Let $M$ be a compact, connected, $C^{1,1}$ embedded surface in $\RR^3$ and let
$\nu$ be the unit normal to $M$ that points out of the compact region bounded by $M$.
We say that $M$ is a {\bf $g$-surface} provided
\begin{enumerate}
\item $M$ is invariant under  reflections in the planes given in cylindrical coordinates
by $\theta = k \pi/g$ for $k=0,1,\dots, g-1$.
\item $\nu(p)\cdot \Theta(p)\le 0$ for all~$p\in M\cap W(0,\pi/g)$,
  with strict inequality for some $p\in M\cap W(0,\pi/g)$.
\end{enumerate}
The $g$-surface is called a {\bf strict $g$-surface}
it is smooth and if $\nu(p)\cdot\Theta(p)<0$ for all $p\in M\cap W(0,\pi/g)$.
\end{definition}

If $M$ is a $g$-surface, then the reflectional symmetries imply that
\[
   \nu(p)\cdot \Theta(p) 
   \begin{cases}
   \le 0   &\text{if $p\in M\cap W(k \pi/g, (k+1)\pi/g)$ with $k$ even,} \\
   \ge 0   &\text{if $p\in M\cap W(k \pi/g, (k+1)\pi/g)$ with $k$ odd}.
   \end{cases}
\]
The inequalities are strict if $M$ is a strict $g$-surface.

\begin{theorem}\label{theta-graph-theorem}
Suppose that $M$ is a strict $g$-surface.  Then
the set
\begin{equation}\label{portion}
  M \cap W(0,\pi/g)  \tag{*}
\end{equation}
is disjoint from its image under $\Rr_\theta$ for $0<\theta<2\pi$.
\end{theorem}

\begin{proof}
Note that each circle of the form
\[
   C = \{(x,y,z): x^2+y^2=r^2, \,  z= \zeta \}
\]
is transverse to~\eqref{portion}.  If the assertion were false, there would be such a circle
that intersects the set~\eqref{portion} in more that one point.  Let $p$ and $q$ be two successive points
in that set.  Let $K$ be the compact region whose boundary is $M$.
Then the arc of $C$ between $p$ and $q$ either lies in $K$ or in $K^c$.
Either way,  $\nu \cdot \Theta$ would be positive at one of the points in $\{p, q\}$
and negative at the other, 
contrary to the definition of strict $g$-surface.
\end{proof}

\begin{corollary}\label{theta-graph-corollary}
Suppose that $M$ is a strict $g$-surface and that $P$ is a plane containing
the $z$-axis.
Let $\Sigma = M\cap W(0,\pi/g)$ and let $\Sigma'$ be the image of $\Sigma$
under reflection in $P$.
Then
\[
   \Sigma\cap\Sigma' \subset P.
\]
\end{corollary}

\begin{proof}
Suppose $p\in \Sigma\cap \Sigma'$.   Since $p\in \Sigma'$, the image $p'$ of $p$
under reflection in $P$ is contained in $\Sigma$.   Since $p$ and $p'$ both lie the circle
\[
   \{ \Rr_\theta p: \theta\in \RR\},
\]
we see that $p=p'$ by Theorem~\ref{theta-graph-theorem}.  Thus $p\in P$.
\end{proof}

\begin{theorem}\label{instant-g-surface-theorem}
Suppose that $M$ is a $g$-surface and let $M(t)$, $t\in [0,T)$
be the associated mean curvature flow.
Then $M(t)$ is a strict $g$-surface for all $t\in [0,T)$.
\end{theorem}

\begin{proof}
Mean curvature flow preserves symmetries, so each $M(t)$ has the symmetries
of a $g$-surface.

Let $\xx: M\times [0,T)\to \RR^3$ be 
 the standard parametrization of the mean curvature flow $M(\cdot)$.
Let $\Sigma$ be the closure of $M\cap W(0,\pi/g)$ and $\Sigma(t)$
be the closure of $M(t)\cap W(0,\pi/g)$.
By symmetry, $\xx(\cdot,t)$ maps $\Sigma$ diffeomorphically onto $\Sigma(t)$.

Let $u(p,t)= \nu(x,t)\cdot \Theta (\xx(p,t))$.
A calculation (see Proposition~\ref{Theta-proposition} below)
shows that $u$ satisfies the evolution equation
\[
   (\partial_t - \Delta)u = |A|^2 u
\]
where $|A(p,t)|$ is the norm of the second fundamental form of $M(t)$ at $\xx(p,t)$.

Now
\begin{align*}
   &u \le 0 \quad \text{on $\Sigma\times\{0\}$, and} \\
   &u \equiv 0 \quad\text{on $(\partial\Sigma)\times (0,T]$}
\end{align*}
Thus  $u\le 0$ on $\Sigma\times(0,T]$ by the maximum principle.

Furthermore, if there were a $t>0$ and a point $p\in \Sigma\setminus \partial \Sigma$
at which $u(p,t)=0$, then, by the strong maximum principle, $u$
would vanish everywhere on $\Sigma\times[0,T]$, a contradiction.
(By definition of $g$-surface, $u<0$ at some point of $\Sigma\times\{0\}$.)
\end{proof}

\begin{lemma}\label{Theta-lemma}
Suppose $\Sigma$ is a smoothly embedded, oriented surface in $\RR^3$.
Let $\nu$ be the unit normal vector field and let $\Theta$ be the map given 
   by~\eqref{Theta-recalled}.
Then
\[
  \Delta (\nu\cdot \Theta) = (\Delta \nu)\cdot \Theta + \nu\cdot (\Delta\Theta).
\]
\end{lemma}

\begin{proof}
Let $p$ be a point in $\Sigma$.  Let $\uu_1$, $\uu_2$ be principal directions
of $\Sigma$ at $p$, and let $\kappa_1, \kappa_2$
be the corresponding principal curvatures.  Extend $\uu_1$ and $\uu_2$ to be a pair of orthonormal 
tangent vectorfields defined in a neighborhood of $p$ such that
\[
  \nabla_{\uu_i} \uu_j(p) = 0.
\]
Then, at the point $p$,
\begin{align*}
\Delta (\nu\cdot\Theta)
-  (\Delta \nu)\cdot \Theta - \nu\cdot(\Delta\Theta) 
&=
2 \textstyle\sum_i (\nabla_{\uu_i} \nu) \cdot (\nabla_{\uu_i} \Theta)   \\
&=
2  \textstyle\sum_i (\kappa_i \uu_i) \cdot \Theta(\uu_i)  \\
&=
0
\end{align*}
since $\Theta:\RR^3\to\RR^3$ is an antisymmetric linear map.
\end{proof}

\begin{proposition}\label{Theta-proposition}
Let $\xx: M\times [0,T)\mapsto \RR^3$ be the standard parametrization
of a mean curvature flow, and let $u= \nu \cdot \Theta$.
Then $u_t - \Delta u = |A|^2u$.
\end{proposition}

\begin{proof}
By Lemma~\ref{Theta-lemma},
\begin{align*}
(\partial_t - \Delta)u
&=
(\partial_t - \Delta) (\nu\cdot \Theta(\xx))  \\
&=
((\partial_t - \Delta)\nu) \cdot \Theta(\xx) + \nu\cdot (\partial_t - \Delta) (\Theta(\xx)) 
    \\
&=
(|A|^2\nu)\cdot \Theta(\xx) + \nu \cdot \Theta((\partial_t - \Delta) \xx) 
 \\
&=
|A|^2 u + 0.
\end{align*}
\end{proof}

\begin{theorem}\label{g-surface-regularity-theorem}
Let $M(t)$, $t\in [0,T)$, be a smooth mean curvature flow where $M(0)$ is a $g$-surface.
Suppose that $(p,T)$ is a singular point of the flow with
\[
p\notin Z = \{(0,0,z): z\in\RR\}.
\]
Then $\theta(p)$ is an integral multiple of $\pi/g$.
\end{theorem}

\begin{proof}
Suppose, to the contrary, that $\theta(p)$ is not an integer multiple of $\pi/g$.
By symmetry, we can assume that $0<\theta(p)<\pi/g$.
Let $M'(t)$ be the image of $M(t)$ under reflection in the plane containing $Z\cup\{p\}$.

By Corollary~\ref{theta-graph-corollary}, there is an open ball $U$ centered at $p$ such that
\[
  M(t)\cap M'(t) \cap U \subset P
\]
for all $t\in (0,T)$.  
Hence, by a Hopf Boundary Type Principle~\cite{choi-h-h-w}*{Theorem~3.19},
$(p,T)$ is a regular point, a contradiction.
  As in the proof of Theorem~\ref{instant-pancake-theorem}, 
  the ``tameness'' hypothesis of Theorem~3.19
is satisfied according to \S\ref{bamler-kleiner-section}.
\end{proof}

\section{$g$-Wheels}\label{wheel-section}

Using the terminology of the last two sections, we can restate the definition
of $g$-wheels (Definition~\ref{wheel-definition}) as follows.
A compact, connected, $C^{1,1}$ embedded surface $M$ in $\RR^3$
is called a $g$-wheel if it is both a pancake and a $g$-surface,
and if $M\cap\{z=0\}$ has exactly $g+1$ components,  one of which lies in $W(-\pi/g,\pi/g)$.

The outermost component $\Gout=\Gout(M)$ of $M\cap\{z=0\}$  has all the symmetries of $M$.
 Let $\Gin=\Gin(M)$ be the component
of $M\cap\{z=0\}$ that lies in $W(-\pi/g, \pi/g)$.  Note that the
curve $\Gout$ together with the curves 
   $\Rr_{\pi/g}^k\Gin$ ($k=0,\dots,g-1$) form a collection of $g+1$ curves in $M\cap\{z=0\}$,
and therefore constitute all of $M\cap\{z=0\}$.
It follows that
\begin{equation}\label{inner-curves}
   (M \setminus \Gout)\cap\{z=0\} \cap \overline{W(-\pi/g,\pi/g)} = \Gin.
\end{equation}

\begin{lemma}\label{x-axis-lemma}
Suppose $M$ is a $g$-wheel, and let $K$ be the compact region
bounded by $M$.  The ray $\{(r,0,0): r\ge 0\}$ intersects
$M$ at in exactly three points $p_1=(r_1,0,0)$, $p_2=(r_2,0,0)$, and $p_3=(r_3,0,0)$,
where $0<r_1<r_2<r_3$, and the ray intersects $M$ orthogonally at those points.
 The points $p_1$ and $p_2$ are in $\Gin$, and $p_3$ is in $\Gout$.
Furthermore, the open line segment $(p_1,p_2)$ from $p_1$ to $p_2$ is in $K^c$, and the closed
line segment $[p_2,p_3]$  from $p_2$ to $p_3$ is in $K$.
\end{lemma}

\begin{proof}
Since $\Gin$ is a simple closed curve that is symmetric about $X$ (the $x$-axis), it intersects
  $X$ in exactly two points.  Let the two points be $p_1=(r_1,0,0)$ and $p_2=(r_2,0,0)$, with $r_1<r_2$.
 Since $\Gin$ lies in $W(-\pi/g,\pi/g)$, 
  it follows that $0<r_1<r_2$.
 Likewise $\Gout$ intersects $X$ in exactly two points $p_0= (r_0,0,0)$ and $p_3=(r_3,0,0)$, where
   $r_0<r_3$.
 Since $\Gout$ is invariant under $\Rr_{\pi/g}$, $\Gout$ winds once around the origin. 
 Thus $r_0< 0$.  Since $\Gin$ lies in the region enclosed by $\Gout$,
it follows  that $r_2< r_3$.
Since $M$ is invariant under $(x,y,z)\mapsto (x,-y,z)$ and under $(x,y,z)\mapsto (x,y,-z)$,
$\nu=\pm \ee_1$ and each of the points $p_i$.   Thus $\nu(p_3)=\ee_1$, 
since $\nu$ points out of the compact region bounded by $M$.  It follows that $\nu(p_2)=-\ee_1$
and $\nu(p_1)=\ee_1$.

Note that the ray $\{(s,0,0): s>r_3\}$ lies in the unbounded component of $M^c$, i.e., in $K^c$.
Thus $[p_2,p_3]$ is contained in $K$ and $(p_1,p_2)$ is contained in $K^c$.
\end{proof}
 
\begin{definition}\label{r-R-definition}
If $M$ is a $g$-wheel, we let
\begin{align*}
r(M) &= \min\{|p|: p\in \Gin\}, \\
R(M) &= \max\{|p|: p\in \Gin \}.
\end{align*}
\end{definition}

Note that $r(M)=r_1$ and $R(M)=r_2$, where $r_1$ and $r_2$ are as
in Lemma~\ref{x-axis-lemma}.

\begin{theorem}\label{basic-g-wheel-theorem}
Suppose that $M$ is a $g$-wheel and let $M(t)$, $t\in[0,T)$, be the associated 
 mean curvature flow.
Then $M(t)$ is a strict $g$-wheel for all $t\in (0,T)$.
Furthermore, if $(p,T)$ is a singular point of the flow, then $z(p)=0$, and, if $p\ne 0$,
then $\theta(p)$ is a  multiple of $\pi/g$.
\end{theorem}

\begin{proof}
This follows immediately from Theorems~\ref{instant-pancake-theorem},
  \ref{pancake-regularity-theorem}, 
  \ref{instant-g-surface-theorem}, and~\ref{g-surface-regularity-theorem}.
\end{proof}

\begin{theorem}\label{wheel-genus-theorem}
Suppose that $M$ is a strict $g$-wheel.  Then
\begin{enumerate}
\item\label{wheel-genus-1} $M$ has genus $g$.
\item\label{wheel-genus-2} $M^*:= M\setminus \Gout(M)$ has genus $g-1$.
\item\label{wheel-genus-3} $\Sigma:=M^*\cap W(\alpha, \alpha+\pi/g)$ 
    has genus $0$ for each $\alpha\in \RR$.
\item\label{wheel-genus-4} If $r(M)<\rho<R(M)$, then
\[
    M':= M\cap \{\dist(\cdot,Z)<\rho\}
\]
has genus $0$,
where $r(M)$ and $R(M)$ are as in Definition~\ref{r-R-definition}.
\end{enumerate}
\end{theorem}

\begin{proof}
Assertion~\eqref{wheel-genus-1} is a special case of Theorem~\ref{pancake-genus-theorem}.
Assertion~\eqref{wheel-genus-2} holds because removing
  a curve from a surface decreases the genus by $1$ if
removing the curve does not disconnect the surface.
To prove Assertion~\eqref{wheel-genus-3}, note
      that the surfaces $\Sigma_k:=\Rr_{\pi/g}^k \Sigma$, where $k=1,2,\dots,g$, 
are disjoint subset of $M^*$.  Thus
\begin{align*}
g-1
&=
\genus(M^*)    \\
&\ge
\sum_{k=1}^g \genus(\Sigma_k)   \\
&=
g \genus(\Sigma),
\end{align*}
so $\genus(\Sigma)=0$.

To prove Assertion~\eqref{wheel-genus-4}, note that
\[
  M' = M^* \cap \{\dist(\cdot,Z)<\rho\}.
\]
Note also that $M^*\cap \{\dist(\cdot,Z)=\rho\}$
(which is the same as $M\cap\{\dist(\cdot,Z)=\rho\}$) consists
of $g$ simple closed curves.

(One of the simple closed curves is in $W(0,2\pi/g)$; the others 
are rotated images of that curve.)

Note also that removing $g-1$ of those curves does not disconnect
$M^*$.  Thus
\[
\genus(M')
\le
\genus(M^*) - (g-1)  
\le 0
\]
by Assertion~\eqref{wheel-genus-2}.
\end{proof}

\begin{corollary}\label{wheel-cylinder-corollary}
Suppose that $M$ is a $g$-wheel and that $M(t)$, $t\in [0,T)$, is the associated
mean curvature flow.  Suppose that
\[
  t\in (-\infty,0)\mapsto |t|^{1/2}\Sigma
\]
is a tangent flow at a singular point $(p,T)$ of the flow with $p\ne 0$ and $p\notin \Gout(T)$.
(See Definition~\ref{Gout-definition}.)
Then $\Sigma$ is a horizontal or vertical cylinder.
\end{corollary}

\begin{proof}
Let $B$ be a an open ball centered at $p$, with radius small enough that $\overline{B}$
is disjoint from $\Gout(T)$ and that
\[
B \subset W( \theta(p)-\pi/(2g),  \theta(p)+\pi/(2g)).
\]
Then for all $t$ sufficiently close to $T$, $B$ is disjoint from $\Gout(t)$.
Thus, for such $t$, $M(t)\cap B$ is contained in 
\[
      (M(t)\setminus \Gout(t)) \cap W(\theta(p) - \pi/g ,\theta(p)+\pi/g),
\]
and hence $M(t)\cap B$ has genus $0$ by Theorem~\ref{wheel-genus-theorem}. 
Hence $\Sigma$ has genus $0$, so $\Sigma$ is a vertical or horizontal shrinking
cylinder by Theorem~\ref{pancake-regularity-theorem}.
\end{proof}

\section{Special $g$-Wheels}\label{special-wheel-section}

Throughout this section, $M$ is a special $g$-wheel
and  $M(t)$, $t\in [0,T)$,
is the associated mean curvature flow.
Recall that $M$ is special means that if $p\in\Gout(T)$,
then $(p,T)$ is regular point of the flow, where $\Gout(T)$ is as in Definition~\ref{Gout-definition}.

Hence, by Corollary~\ref{distance-corollary}, if $(p,T)$ is a singular point of the flow,
then
\begin{equation}\label{inner-close}
   \limsup_{t\to T} (T-t)^{-1/2} \dist(p, \Gamma(t) \setminus \Gout(t))  < \infty,
\end{equation}
where $\Gamma(t):= M(t)\cap\{z=0\}$.

\begin{lemma}\label{origin-lemma}
Suppose $(0,T)$ is a spacetime singularity of the flow.
Then
\begin{enumerate}
\item\label{R-bounded} $\limsup_{t\to T} (T-t)^{-1/2} R(t) < \infty$, 
where 
\[
  R(t) := \max\{ |p|:  p\in \Gamma(t) \setminus \Gout(t)\}.
\]
\item\label{only-item} $(0,T)$ is the the only singularity of the flow.
\item\label{origin-3} If $\Sigma$ is the shrinker arising from a tangent flow at $(0,T)$,
 then $\Sigma \cap\{z=0\}$ consists of $g$ simple closed curves, and $\Sigma$
 has genus $g-1$.
 Furthermore, $\Sigma^+:=\Sigma\cap\{z>0\}$ projects diffeomorphically
    onto the unbounded component of $\{z=0\}\setminus\Sigma$.
\end{enumerate}
\end{lemma}

\begin{proof}
To prove Assertion~\eqref{R-bounded}, suppose to the contrary that there exist $t_i\to T$
such that
\begin{equation}\label{stretched}
   (T-t_i)^{-1/2}R(t_i)  \to \infty.
\end{equation}
Translate in time by $-T$ and then dilate parabolically
by $(p,t)\mapsto ((T-t_i)^{-1/2}p, (T-t_i)^{-1}t)$
to get a flow $M_i(\cdot)$.
After passing to a subsequence, then $M_i(\cdot)$ converge
to a tangent flow
\[
   t\in (-\infty,0)\mapsto |t|^{1/2}\Sigma.
\]
(See~\cite{ilmanen-sing}*{Lemma~8} 
for existence of such a subsequence.)
By Theorem~\ref{wheel-genus-theorem}\eqref{wheel-genus-4}, $\Sigma$ has genus $0$
and thus it is a cylinder.
Since $M$ is invariant under $\Rr_{2\pi/g}$, so is $\Sigma$.
Thus $\Sigma$ must be a vertical cylinder.
But that is impossible since
 $\Sigma\cap\{z=0\}$
contains no closed curves (by~\eqref{stretched}).
Thus, Assertion~\eqref{R-bounded} holds.

By Assertion~\eqref{R-bounded}, $\Gamma(t)\setminus \Gout(t)$ converges to $0$
as $t\to T$.  Hence, by~\eqref{inner-close}, there are no 
other singular points.

To prove Assertion~\eqref{origin-3}, note that there is a sequence $t_i\to T$ such that
\[
    (T-t_i)^{-1/2} M(t_i) \to  \Sigma
\]
smoothly.
Hence, by Assertion~\eqref{R-bounded} and by specialness, $\Sigma\cap\{z=0\}$
consists of $g$ simple closed curves.  
Thus, $\Sigma$ has genus $g-1$ by Corollary~\ref{pancake-shrinker-genus-corollary}.
The last statement about $\Sigma^+$ holds
by Theorem~\ref{pancake-shrinker-theorem}\eqref{p-shrinker-graph}.
\end{proof}

\begin{corollary}\label{origin-corollary}
If $(0,T)$ is a spacetime singular point, then all the inner components
of $M(t)\cap\{z=0\}$ converge to $0$ as $t\to T$.
\end{corollary}

The corollary is a trivial consequence of 
  Theorem~\ref{origin-lemma}\eqref{R-bounded}.

\begin{lemma}\label{vertical-lemma}
Suppose that $(p,T)$ is a spacetime singular point.
The following are equivalent:
\begin{enumerate}
\item\label{v-lemma-1} $(p,T)$ has a shrinking vertical cylinder as a tangent flow.
\item\label{v-lemma-2} $p\ne 0$ and $\theta(p)$ is an even multiple of $\pi/g$.
\item\label{v-lemma-3}  All tangent flows at time $T$ singularities are shrinking vertical cylinders.
\end{enumerate}
Furthermore, if these equivalent statements hold, then there are exactly $g$ singular points at time $T$,
namely the points $\Rr_{2\pi/g}^kp$, where $k=0,1,\dots, g-1$,
and each of the inner components of $M(t)\cap\{z=0\}$ collapses to one of those
points as $t\to T$.
\end{lemma}

\begin{proof}
Suppose that~\eqref{v-lemma-1} holds.  
Recall that $p$ lies in the plane $\{z=0\}$.
By symmetry, we can assume that
\[
  p \in \overline{W(-\pi/g,\pi/g)}.
\]
There exist $t_i\to T$ and nearly circular curves $C_i$ in $M(t_i)\cap \{z=0\}$
that enclose $p$ and that collapse to $p$ as $i\to\infty$.
Thus for large $i$, $C_i\ne \Gout(t_i)$, as $\Gout(t_i)$ converges to the simple
closed curve $\Gout(T)$.  Thus $C_i$ is one of the inner components
  of $M(t_i)\cap \{z=0\}$.  None of the inner components of $M(t_i)\cap\{z=0\}$
  intersect $\partial W(-\pi/g,\pi/g)$ (by~\eqref{inner-curves}).
  Therefore $C_i$ is contained in $W(-\pi/g,\pi/g)$,
  and hence $p$ is contained in $W(-\pi/g,\pi/g)$.   Since $\theta(p)$ is a multiple
  of $\pi/g$, in fact $\theta(p)=0$.
Thus we have proved that~\eqref{v-lemma-1} implies~\eqref{v-lemma-2}.

Now suppose that~\eqref{v-lemma-2} holds.
By symmetry, we can assume that $\theta(p)=0$.
Let 
\[
  t\in (-\infty,0) \mapsto |t|^{1/2} \Sigma
\]
be a tangent flow.  Note the unit normal $\nu$ on $\Sigma$
satisfies
\begin{align}
\nu\cdot \ee_3&\ge 0 \quad\text{on $\Sigma^+:=\Sigma\cap\{z>0\}$, and}
\label{nu-e3} 
\\
\nu\cdot \ee_2&\le 0 \quad\text{on $\Sigma\cap\{y>0\}$},
\label{nu-e2}
\end{align}
since $M(t)$ is a pancake and since it is a $g$-surface.
By Corollary~\ref{wheel-cylinder-corollary}, $\Sigma$ is a cylinder.  By symmetry, the axis of the cylinder
is one of the coordinate axes.
If the axis of $\Sigma$ were the $x$-axis, then $\nu$ would point out of the cylinder by~\eqref{nu-e3}
and into the cylinder by~\eqref{nu-e2}, a contradiction. Thus the axis is not the $x$-axis.

If the axis of $\Sigma$ were the $y$-axis,
 then $\Sigma$ would intersect the $x$-axis
orthogonally in the two points $\pm q = (\pm \sqrt2,0,0)$.  By~\eqref{nu-e3}, $\nu$ points
out of the cylinder.  Thus $\nu(-q)= - \ee_1$ and $\nu(q)=\ee_1$.
But $-q$ and $q$ correspond to the points $p_1$ and $p_2$ in Lemma~\ref{x-axis-lemma},
so $\nu(-q)=\ee_1$ and $\nu(q)=-\ee_1$, a contradiction.
(Note that $q$ does not correspond to $p_3$ in Lemma~\ref{x-axis-lemma}, because $\Gout(t)$ stays
away from $p$ as $t\to T$.)

Hence the unique tangent flow at $(p,T)$ is the vertical cylinder.

It follows (see~\eqref{inner-curves}) that
\[
    (M(t)\cap\{z=0\})\setminus \Gout(t)\cap \overline{W(-\pi/g,\pi/g)} 
\]
converges to the point $p$.
By symmetry, 
\[
   (M(t)\cap\{z=0\})\setminus \Gout(t) 
\]
converges to the points $\Rr_{\pi/g}^kp$ (where $k=1,2,\dots,g$).
At each of those points at time $T$, the tangent flow is a vertical cylinder, and there
are no other singularities of the flow (by Corollary~\ref{distance-corollary}).
Thus~\eqref{v-lemma-2} implies~\eqref{v-lemma-3}, as well as the ``furthermore''
assertion.

Finally, \eqref{v-lemma-3} trivially implies~\eqref{v-lemma-1}.
\end{proof}

\begin{lemma}\label{horizontal-lemma}
Suppose that $(p,T)$ is a spacetime singular point.
The following are equivalent:
\begin{enumerate}
\item\label{h-lemma-1} $(p,T)$ has a shrinking horizontal cylinder as a tangent flow.
\item\label{h-lemma-2} $p\ne 0$ and $\theta(p)$ is an odd multiple of $\pi/g$.
\item\label{h-lemma-3}  All tangent flows at time $T$ singularities are horizontal shrinking cylinders.
\end{enumerate}
Furthermore, if these hold, then the axis of the shrinking cylinder
  coming from the singularity at $(p,T)$ is the line $\{cp: c\in\RR\}$.
\end{lemma}

\begin{proof}
Suppose that~\eqref{h-lemma-1} holds. By Lemma~\ref{origin-lemma}, $p\ne 0$.
By Theorem~\ref{basic-g-wheel-theorem}, $z(p)=0$ and $\theta(p)$ is a multiple of $\pi/g$.  
By Lemma~\ref{vertical-lemma}, $\theta(p)$ is not an even multiple of $\pi/g$.
Hence it is an odd multiple.  Hence~\eqref{h-lemma-1} implies~\eqref{h-lemma-2}.

Now suppose that~\eqref{h-lemma-2}  holds.  By Lemma~\ref{origin-lemma}, $(0,T)$ is not a singular point.
By Corollary~\ref{wheel-cylinder-corollary},
   each tangent flow at $(p,T)$ is a vertical cylinder or a horizontal cylinder.
By Lemma \ref{vertical-lemma}, it cannot be a vertical cylinder.
Thus it is a horizontal cylinder.  
Again, by Lemma~\ref{vertical-lemma}, no tangent flows at time $T$
are vertical shrinking cylinders.  Thus (since $(0,T)$ is not a singular point),
all tangent flows at time $T$ singularities are horizontal shrinking cylinders.
Thus~\eqref{h-lemma-2} implies~\eqref{h-lemma-3}.

Trivially,~\eqref{h-lemma-3} implies~\eqref{h-lemma-1}.

Now suppose that~\eqref{h-lemma-1}, \eqref{h-lemma-2}, and~\eqref{h-lemma-3} hold.
Note that the line $L=\{cp: c\in\RR\}$ only intersects $M(t)$ at the points of $\Gout(t)\cap L$,
and those points are bounded away from $p$ as $t\to T$ (since $M$ is special).
Thus if
\[
  t\in (-\infty,0) \mapsto |t|^{1/2}\Sigma
\]
is a tangent flow at $(p,T)$, then $L$ does not intersect $\Sigma$.
Thus $L$ is the axis of the cylinder $\Sigma$.
\end{proof}

\newcommand{\Qodd}{Q_\textnormal{odd}} 
\newcommand{\Qeven }{Q_\textnormal{even}} 

\begin{definition}
We let $\Qodd$ be the set  consisting
of the origin together with the points $p$ in $\{z=0\}\setminus\{0\}$ such that
 $\theta(p)$ is an odd multiple of $\pi/g$.
We let $\Qeven$ be the set consisting of the origin
together with the points $p$ in $\{z=0\}\setminus\{0\}$ such that $\theta(p)$ is an even multiple of $\pi/g$.
\end{definition}

Thus each of $\Qodd$ and $\Qeven$ is a  union of $g$ horizontal rays emanating from the origin

\begin{theorem}[Trichotomy Theorem]\label{2nd-trichotomy-theorem}
Suppose that $M$ is a  special $g$-wheel and that $M(t)$, $t\in [0,T)$,
is the associated mean curvature flow.
 Then exactly one
of the following happens:
\begin{enumerate}
\item\label{tri-1} All of the tangent flows at time $T$ singularities are horizontal cylinders.
   In this case, all the singular points lie in $\Qodd\setminus\{0\}$.
\item\label{tri-2} All of the tangent flows at time $T$ singularities are vertical cylinders.
  In this case, there are exactly $g$ singularities, and they all lie in $\Qeven\setminus\{0\}$.
\item\label{tri-3} The origin is the only singular point at time $T$, and for each tangent flow
\[ 
   t\in (-\infty,0)\mapsto |t|^{1/2}\Sigma,
\]
the shrinker $\Sigma$ has genus $g-1$, and $\Sigma\cap \{z=0\}$ consists of $g$ simple closed curves.
Furthermore, $\Sigma^+:=\Sigma\cap\{z>0\}$ projects diffeomorphically
    onto the unbounded component of $\{z=0\}\setminus\Sigma$.
\end{enumerate}
\end{theorem}

\begin{proof}
Let $(p,T)$ be a singular point of the flow.
  By Theorem~\ref{basic-g-wheel-theorem}, one of the following holds:
\begin{enumerate}[\upshape(i)]
\item $p\ne 0$ and $\theta(p)$ is an odd multiple of $\pi/g$.
\item $p\ne 0$ and $\theta(p)$ is an even multiple of $\pi/g$.
\item $p=0$.
\end{enumerate}
By Lemmas~\ref{horizontal-lemma}, \ref{vertical-lemma}, and~\ref{origin-lemma},
 these correspond to cases~\eqref{tri-1}, \eqref{tri-2}, and~\eqref{tri-3} of the Trichotomy Theorem.
\end{proof}

Recall that a special $g$-wheel $M$ is called {\bf thin}, {\bf thick},
or {\bf critical} according to whether (1), (2), or (3) holds in Theorem~\ref{2nd-trichotomy-theorem}.

\begin{theorem}\label{open-theorem}
The set of thin $g$-wheels and the set of thick $g$-wheels
are open subsets of the set of all special $g$-wheels.
\end{theorem}

\begin{proof}
Let $M_i$ be a sequence of special $g$-wheels that are not thin
and that converge to a special $g$-wheel $M$.  Let $M_i(t)$, $t\in [0,T_i)$,
and $M(t)$, $t\in [0,T)$, be the associated mean curvature flows.

For each $i$, let $(p_i,T_i)$ be a spacetime singular point of the flow $M_i(\cdot)$. 
 Then, by hypothesis,
for each $i$, $p_i\in Q_{\rm odd}$.  
After passing to a subsequence, $p_i$ converges to a point $p\in Q_{\rm odd}$.
By local regularity theory~\cite{white-local}, $(p,T)$ is a singular point of $M(\cdot)$.
Since $p\in Q_{\rm odd}$, $M$ is not thin.
Thus we have proved openess of the set of thin $g$-wheels.

Openness of the set of thick $g$-wheels is proved in exactly the same way.
\end{proof}

\begin{lemma}\label{angenent-lemma}
Suppose that $M$ is a smoothly embedded surface in $\RR^3$ and that
  $A$ is an Angenent torus in $M^c$.
Let $M(t)$, $t\in [0,T_M)$, and $A(t)$, $t\in [0,T_A)$, be the mean curvature flows associated
to $M$ and to $A$.   If $T_A\le T_M$, then $A$ is homotopically trivial in $M^c$.
\end{lemma}

\begin{proof}
It suffices to show that $A(t)$ is homotopically trivial in $M(t)^c$ for some $t\in [0,T_A)$.
Let 
\[
  \eps = \dist(M,A).
\]
Then
\[
\text{$0 < \eps \le \dist(M(t),A(t))$ for all $t\in [0,T_A)$.}
\]
Let $p$ be the point such that $A(t)\to\{p\}$ as $t\to T_A$.
Choose $t<T_A$ close enough to $T_A$ that $A(t)$ is contained in the ball $B(p,\eps/3)$.
Then $M(t)$ lies outside $B(p,2\eps/3)$.  Since $A(t)$ is homotopically trivial in $B(p,\eps/3)$,
it is homotopically trivial in $M(t)^c$.
\end{proof}

The following two theorems give sufficient conditions for a special $g$-wheel to be thin or thick.

\begin{theorem}[Thinness Condition]\label{thin-theorem}
Suppose that $M$ is a special $g$-wheel.
Suppose also that there is an Angenent torus $A$ in $M^c$ such that the axis of $A$
   is a line in $\{z=0\}$ and such that one of the inner components of $M(\tau) \cap\{z=0\}$
   is homotopically nontrivial in $A^c$.
Then $M$ is thin.
\end{theorem}

\begin{proof}
Let $M(t)$, $t\in [0,T)$, and 
Let $A(t)$, $t \in [0,T_A)$, be the mean curvature flows associated to $M$ and to $A$.
By Lemma~\ref{angenent-lemma}, $T<T_A$.   Let $C(t)$ be an inner component of $M(t)\cap\{z=0\}$
such that $C(0)$ is homotopically nontrivial in $A(0)^c$.
Thus $C(t)$ is homotopically nontrivial in $A(t)^c$ for all $t\in [0,T)$.
Thus $C(t)$ enclosed one of the components of $A(t)\cap \{z=0\}$, so
\[
   \diam C(t) \ge d_A(t) \ge d_A(T) > 0,
\]
where $d_A(t)$ is the diameter of either of the two components of $A(t)\cap\{z=0\}$.
(The two components are congruent.)
Hence $C(t)$ does not shrink to a point as $t\to T$.
Thus $M$ is neither thick nor critical by Lemma~\ref{vertical-lemma}
 and Corollary~\ref{origin-corollary}.  Hence $M$ is thin.
\end{proof}

\begin{definition}
If $A$ is an Angenent torus centered at a point in $\{z=0\}$ and with a vertical axis,
then $\{z=0\}\setminus A$ has three components: an unbounded component,
an annulus, and a disk.  We let $D_A$ denote the disk.
\end{definition}

\begin{theorem}[Thickness Condition]\label{thick-theorem}
Suppose that $M$ is a special $g$-wheel.
Suppose also that there is an Angenent torus $A$ such that $A$ has vertical axis,
  $A$ is centered at a point $(x,0,0)$ with $x>0$, $A$ is in $W(-\pi/g,\pi/g)$, 
  $A$ is in the interior of the region $K$ bounded by $M$, and
  $D_{A}$ contains the component of $M\cap\{z=0\}$ lying in $W(-\pi/g, \pi/g)$.
Then $M$ is thick. 
\end{theorem}

\begin{proof}
Let $M(t)$, $t\in [0,T)$ and $A(t)$, $t\in [0,T_A)$, be
the mean curvature flows associated to $M$ and to $A$.
Let $C(t)$ be component of $M\cap\{z=0\}$ in $W(-\pi/g, \pi/g)$.
By hypothesis, $C(0)$ is contained in $D_{A}=D_{A(0)}$. 
Let $V$ be a vertical line that passes through the region in $\{z=0\}$ enclosed by $C(0)$.
Then $V$ is contained in $M^c$.  Since $A$ is not homotopically trivial in $V^c$, 
it is not homotopically trivial in $M^c$.   Thus $T_A>T$ by Lemma~\ref{angenent-lemma}.

Since $C(0)$ is contained in $D_{A(0)}$, it
follows that $C(t)$ is contained in $D_{A(t)}$, and therefore that
\begin{equation}\label{Gin-inside}
   C(t)\subset D_{A(0)}=D_A.
\end{equation}
for all $t$.  Let $p$ be a point such that $(p,T)$
is a singular point of the flow $M(\cdot)$.  By symmetry, we can
assume that $p$ is the closure of $W(-\pi/g, \pi/g)$.  
For each $t$, by symmetry, $\Gin(t)$ is the inner component of $M(t)\cap\{z=0\}$
to which $p$ is closest.  Thus 
\[
   \dist(p,\Gin(t)) \to 0 \quad\text{as $t\to T$}.
\]
Thus, by~\eqref{Gin-inside}, 
\[
   p \in \overline{D_A} \subset W(-\pi/g,\pi/g).
\]
Hence $p\ne 0$ and $\theta(p)$ is not an odd multiple of $\pi/g$,
so $M$ is thick according to Theorem~\ref{2nd-trichotomy-theorem}.
\end{proof}

\section{Existence}\label{existence-section}

Let $\MM_g(R)$ be the collection of all $M(\eps,s;R)$ in Example~\ref{example}.

\begin{lemma}\label{special-lemma}
There is a $\tilde R$ with the following property.  If $R\ge \tilde R$ 
and if $M\in \MM_g(R)$, then $M$ is special.
\end{lemma}

This is a fairly direct consequence of local regularity theory.
  For the reader's convenience,
we include a proof.

\begin{proof}
First, note if $M\in \MM_g(R)$, if $M(t)$, $t\in [0,T_M)$ is the associate
mean curvature flow, and $p\in \{z=0\}$ and if $\dist(p,0)> 1+R$, then the point in $M$ closest
to $p$ is in $\Gout(M)$.  
Recall that if $S\subset M^c$, then
\begin{equation}\label{distance-bound}
   \dist(S, M(t))^2 \ge \dist(S,M)^2 - 4t
\end{equation}
for all $t\ge 0$.
Let
\[
 C_{1+R} = \{(x,y,0): (x^2+y^2)^{1/2}= 1+ R\}.
\]
Since $\dist(M(0),C_{1+R}) > R$, 
\[
\dist(M(t),C_{1+R})^2 
 \ge    \dist(C_{1+R},M(0))^2 - 4t = R^2 - 4t,
\]
so
\begin{align*}
\dist(M(t), C_{1+R}) 
&\ge 
R(1 - 4R^{-2}t)^{1/2}  \\
&\ge
R - 2R^{-1}t 
\end{align*}
Consequently,
\begin{equation}\label{closest}
\begin{gathered}
\text{If $t< R^2/2$ and $p\in\Gout(t)$, then}  \\
\text{the point in $M$ closest to $p$ is in $\Gout(M)$.}
\end{gathered}
\end{equation}

Suppose the lemma is false.
Then there exist $R_i\to \infty$,  $M_i\in\MM(R_i)$ with associated mean curvature flows $M_i(t)$, $t\in [0,T_i)$,
and singular points $(p_i,T_i)$ with $p_i\in {\Gout}_i(T_i)$.
By~\eqref{distance-bound} (with $S=\{p_i\}$),
\begin{equation}\label{distance-i-bound}
   \dist(p_i,M_i)^2 \le 4 T_i.
\end{equation}

Let $A$ be the Angenent torus centered at $0$ and with axis $Z$ such that $\dist(0,A)=1$.
Let $A(t)$, $t\in [0,T_A)$ be the associated mean curvature flow.
We can assume that the $R_i$ are large enough that $A$ is contained in the region
bounded by $M_i$.   Thus
\begin{equation}\label{Ti-Ta}
   T_i < T_A
\end{equation}
by Lemma~\ref{angenent-lemma}.

Let $q_i$ be  point in $M_i$ closest to $p_i$.  By~\eqref{closest} and~\eqref{Ti-Ta},
\begin{equation}\label{qi-gout}
   q_i \in \Gout(M_i)
\end{equation}
for all sufficiently large $i$.   By~\eqref{distance-i-bound},
\begin{equation}\label{pi-qi-close}
   |q_i-p_i|^2 \le 4T_i.
\end{equation}

Translate the flow $M_i(\cdot)$ spatially by $-p_i$ and the dilate parabolically
by 
\[
  (q,t)\mapsto  (q/ T_i^{1/2},  t/T_i)
\]
to get a flow $M_i'(t)$, $t\in [0,1)$ such that the spacetime point $(0,1)\in \RR^3\times\RR$
is a singular point.  Let $q_i'=(T_i)^{-1/2}(q_i-p_i)$ be the point in $M_i'(0)$ corresponding
to  $q_i\in M_i(0)$.   By~\eqref{pi-qi-close},
\[
   |q_i'|\le 2.
\]
After passing to a subsequence, the $q_i'$ converge to a point $q'$ and $M_i'(0)$
converges smoothly and with multiplicity one to a plane $P$ through $q'$.

By passing to a subsequence, the $M_i'(\cdot)$ converge as Brakke flows
to an integral Brakke flow $M'(\cdot)$.  Since $M'(0)$ is a multiplicity $1$ plane and since
   the spacetime point $(0,1)\in \RR^3\times \RR$ is in the support of the flow,
   $M'(t)$ is $P$ with multiplicity one for all $t\in [0,1)$.   
   By the local regularity theory~\cite{white-local}, the convergence $M_i'(\cdot)$ to $M'(\cdot)$
   is uniformly smooth on compact subsets of $\RR^3\times (0,1]$.  But that is a contradiction,
   since $(0,1)$ is a singular point of $M_i'(\cdot)$.
\end{proof}

\begin{theorem}
For each $g\ge 3$, there exists a critical wheel $M_g$.
\end{theorem}

\begin{proof}
Choose $R$ large enough (see Lemma~\ref{special-lemma})
that all the surfaces in $\MM_g(R)$
are special $g$-wheels.
Let $G_\textnormal{thin}$ be the set of thin $g$-wheels in $\MM_g(R)$
and let $G_\textnormal{thick}$ be the set of thick
 $g$-wheels in $\MM_g(R)$.
If $\eps>0$ is sufficiently small, then $M(\eps,s)$ is thin by the thinness condition, 
 Theorem~\ref{thin-theorem}.
If $\eps$ is sufficiently close to $\eta$ and $s>0$ is sufficiently small, then
$M(\eps,s)$ is thick by the thickness condition, Theorem~\ref{thick-theorem}.
By Theorem~\ref{open-theorem},
   $G_\textnormal{thin}$ and $G_\textnormal{thick}$ are open.
Since $G_\textnormal{thin}$ and $G_\textnormal{thick}$ 
are nonempty, disjoint, open subsets of $\MM_g(R)$
and since $\MM_g(R)$ is connected, there must be a critical $g$-wheel $M_g$
in $\MM_g(R)$ by the Trichotomy Theorem~\ref{trichotomy-theorem}.
\end{proof}

\begin{remark}\label{g2-remark}
The case $g=2$ is somewhat different, because (in that case)
it is possible to have a shrinking horizontal cylinder singularity at the origin.
(A horizontal cylinder can have the symmetries of a $g$-wheel if $g=2$, but not
if $g>2$.)  For that reason, Lemma~\ref{horizontal-lemma}
and Theorem~\ref{2nd-trichotomy-theorem} are not true as stated for $g=2$.
The trichotomy theorem as stated in \S\ref{overview-section}
 (Theorem~\ref{trichotomy-theorem}) remains true for $g=2$, and if thick, thin, and critical are defined
 according to Theorem~\ref{trichotomy-theorem} rather than according to 
   Theorem~\ref{2nd-trichotomy-theorem}, 
 then critical $2$-wheels do exist, but somewhat different proofs are required.
 \end{remark}

\section{Behavior as $g\to\infty$}\label{behavior-section}

In this section $\Ss_g$ denotes the family of all smooth shrinkers $\Sigma$ such that
\[
   t\in (-\infty,0)\mapsto |t|^{1/2}\Sigma
\]
is a tangent flow at the spacetime singularity of the mean curvature flow associated
to a critical $g$-wheel.

\begin{lemma}\label{shape-lemma}
If $\Sigma\in \Ss_g$, then $\Sigma^+:=\Sigma\cap\{z>0\}$
is the  graph of a smooth function~$f$, and 
\begin{equation}\label{shape-inequalities}
\begin{aligned}
\pdf{}\theta f(r\cos\theta,r\sin\theta) &\ge 0   \quad (0\le \theta \le \pi/g),
\\
\pdf{}\theta f(r\cos\theta, r\sin\theta) &\le 0  \quad (-\pi/g \le \theta \le 0),
\end{aligned}
\end{equation}
provided $(r\cos\theta, r\sin\theta)$ is in the domain of $f$.
\end{lemma}

\begin{proof}
That $\Sigma^+$ is a smooth graph was proved in Theorem~\ref{2nd-trichotomy-theorem}.
For any $g$-wheel (or, more generally, for any $g$-surface),
\[
   \nu\cdot\Theta\le 0 
\]
in the wedge $W(0,\pi/g)$. 
(See Definition~\ref{g-surface-definition}.)
The same holds for $\Sigma$, since $\Sigma$ is a smooth limit of $g$-wheels.
By symmetry,
\[
  \nu\cdot\Theta\ge 0
\]
in the wedge $W(-\pi/g,0)$.
The inequalities~\eqref{shape-inequalities} follow immediately.
\end{proof}

If $\Sigma$ is a shrinker, then
\[
   t\in (-\infty,0)\mapsto |t|^{1/2}\Sigma
\]
is a mean curvature flow, so $C(\Sigma):=\lim_{t\to 0}|t|^{1/2}\Sigma$ exists.   
Note that $C(\Sigma)$ is a cone, i.e., that $\lambda\Sigma=\Sigma$ for all $\lambda>0$.

We let
\begin{equation}\label{r-R-Sigma}
\begin{aligned}
r(\Sigma)&= \min \{|p|: p\in \Sigma\cap\{z=0\} \}, \\
R(\Sigma)&= \max  \{|p|: p\in \Sigma\cap\{z=0\} \}.
\end{aligned}
\end{equation}

Let $\bar{r}$ be the largest $r$ such that the disk
\[
   D(r):= \{p:|p|\le r\} \cap \{z=0\}
\]
is stable for the shrinker metric, and let
 $\bar{R}$ be the smallest $R$ such that
\[
   \{z=0\}\setminus D(R)
\]
is stable.  Since $\{z=0\}$ is unstable, we see that $\bar{r}<\infty$ and $\bar{R}>0$.

\begin{lemma}\label{r-R-bounds}
If $\Sigma\in \Ss_g$, then $r(\Sigma)<\bar{r}$ and $R(\Sigma)>\bar{R}$. 
\end{lemma}

\begin{proof}
Let $\Omega$ be the unbounded component of $\{z=0\}\setminus \Sigma$.
By Theorem~\ref{stable-theorem}, 
 $\Omega$ is stable (with respect to the shrinker metric~\eqref{shrinker-metric}).
Since $\Omega$ contains $D(r(\Sigma))$, we see that $D(r(\Sigma))$ is strictly
stable.  Thus $r(\Sigma)<\bar{r}$.
Likewise $\{z=0\} \setminus D(R(\Sigma))$ is strictly stable, so $R(\Sigma)>\bar{R}$.
\end{proof}

\begin{lemma}\label{slice-lemma}
Let $\Sigma\in \Ss_g$, and let
\begin{gather*}
\sigma: (-\infty, -R] \cup [-r, r] \cup [R,\infty) \to \RR, \\
\sigma(x)= f(x,0),
\end{gather*}
where $r=r(\Sigma)$ and $R=R(\Sigma)$ (see~\eqref{r-R-Sigma}),
and where $f$ is the function 
whose graph is $\Sigma\cap\{z>0\}$ 
  (as in Theorem~\ref{pancake-shrinker-theorem}).
Then $\sigma$ is strictly decreasing on $[0,r]$ and strictly increasing 
on $[R,\infty)$.
\end{lemma}

\begin{proof}
First we claim that $\sigma$ has no nonzero interior local minima. 
For suppose that $r_0$ was a nonzero interior local minimum.  We may suppose $r_0 >  0$.
Then, by Lemma~\ref{shape-lemma}, the point $(r_0,0)$ is a local minimum
of the function $f$ in that lemma.  Thus, at the point $p=(r_0, 0, f(r_0,0))$,  
\[
    H\cdot \nu = H\cdot\ee_3 \ge 0.
\]
But, since $\Sigma$ is a shrinker, $H = (-p)^\perp/2$ and thus $H\cdot\ee_3< 0$, 
a contradiction.
Hence $\sigma$ has no interior local minima, as claimed.

Since $\sigma(0)>0$ and $\sigma(r)=0$, it follows that $\sigma$ is strictly
decreasing on $[0,r]$.
Since $\sigma(r)=0$ and $\lim_{x\to\infty}\sigma(x) =\infty$ 
 (by Theorem~\ref{pancake-shrinker-theorem}\eqref{p-shrinker-cone-2}),
 it follows that $\sigma$ is strictly increasing on $[R,\infty)$.
\end{proof}

\begin{theorem}\label{flat-shrinker-theorem}
If $\Sigma_g\in \Ss_g$, then $\Sigma_g$ converges to plane $\{z=0\}$
as $g\to\infty$.
After passing to a subsequence,  the convergence is smooth with multiplicity $2$ away from
the circle 
\[
 C = C(r):=  \{p: \text{$z(p)=0$ and  $|p|=r$}\},
\]
for some $r\in(0,\infty)$.
Furthermore, the $\Sigma_g$ converge as varifolds to $\{z=0\}$ with multiplicity $2$.
\end{theorem}

\begin{proof}
By passing to a subsequence, we can assume that $r(\Sigma_g)$ and $R(\Sigma_g)$
converge to limits $r\le R$, where $r\in [0,\infty)$ and $R\in (0,\infty]$ by 
  Lemma~\ref{r-R-bounds}.
Let
\[ 
   A(r,R) = \{ p: z(p)=0, \, r\le |p|\le R\}.
\]
By a straightforward curvature estimate (see Theorem~\ref{curvature-bound-theorem} below),
  we can assume that the $\Sigma_g$ converge smoothly
in $\RR^3\setminus A(r,R)$ to a $\gamma$-minimal lamination $\Sigma$ of $\RR^3\setminus A(r,R)$, where $\gamma$ is the shrinker
metric~\eqref{shrinker-metric}.

\begin{claim}\label{no-vertical-claim}
 No leaf of $\Sigma$ is contained in a vertical plane $P$ containing $Z$.
\end{claim}

For suppose to the contrary that there is a leaf $L$ and a plane $P$ such that
$L\cup Z\subset P$.   Then $L^+:=L\cap \{z>0\}$ is all of $P^+:=\{z>0\}$.
By symmetry, $\Rr_\theta L$ is also a leaf for each $\theta$.
But for $0<\theta<\pi$, $L$ and $\Rr_\theta L$ are distinct leaves whose
intersection is nonempty (the intersection is $Z^+$), a contradiction.
This proves Claim~\ref{no-vertical-claim}.

\begin{claim}\label{rotational-claim} Each leaf of $\Sigma$ is rotationally invariant about $Z$.
\end{claim}

For $p\in Z^c$, $P_p$ be the plane containing $Z\cup \{p\}$.
Let $L$ be a leaf.  Since $L\cup Z$ is not contained in a vertical plane,
there is a curve $\alpha:I \subset \RR\to L\setminus Z$ such that
$\theta\circ \alpha$ is strictly monotonic (where $\theta$ is the cylindrical
coordinate.)  
The symmetries of the $\Sigma_g$ imply that for each $t\in I$,
the image of $L$ under reflection in $P_{\alpha(t)}$ is also a leaf.
Since those leaves intersect (they both contain the point $\alpha(t)$)
they must be the same.  That is, for each $t\in I$, $L$ is invariant under
reflection in $P_{\alpha(t)}$.  It follows that there is an $\omega>0$
such that $\Rr_\theta L=L$ for all $0\le \theta<\omega$.  But then $\Rr_\theta L = L$
for all $\theta$.  Thus Claim~\ref{rotational-claim} is proved.

\begin{claim}\label{r-R-claim} $0< r=R < \infty$.
\end{claim}

Suppose, contrary to the claim, that $r<R$.
Note that if $r<a<b<R$, then
\begin{equation}\label{disjoint}
  \text{$[a,b]\times \{0\}\times \RR$ 
is disjoint from $\Sigma_g$ for all sufficiently large $g$.}
\end{equation}
If some leaf $L$ of $\Sigma$ contained a point $p$ with 
\begin{equation}\label{a-b-cylinder}
   a < \dist(p,Z)<b,
\end{equation}
then, by rotational symmetry, $L$ would contain point in the strip $(a,b)\times\{0\}\times \RR$.
By rotational symmetry, $L$ would intersect that strip orthogonally.  But then, by smooth convergence,
$\Sigma_g$ would intersect the strip for  large $g$, contrary to~\eqref{disjoint}.
Thus $\Sigma$ contains no point $p$ satisfying~\eqref{a-b-cylinder}.
 Since this is true for all $r<a<b< R$, we see that
\[
  \Sigma \cap   \{p: r< \dist(p,Z) < R\} = \emptyset.
\]
Now let $r < a < b < R$, 
and let
\[
    \Aa_g : = \Sigma_g \cap \{p: a<\dist(p,Z)<b\} \cap W(0, 2\pi/g).
\]
Note that $\Aa_g$ is a $\gamma$-minimal annulus with one boundary
component in the cylinder $\{\dist(\cdot,Z)=a\}$ and the other boundary
component in the cylinder $\{\dist(\cdot,Z)=b\}$.

As $g\to\infty$, the annuli $\Aa_g$ converge as subsets of $\{a<\dist(\cdot,Z)<b\}$
to the line segment 
\[
     I: = \{ (x,0,0): a<x<b\}.
\]
But that is a contradiction; such long, thin minimal surfaces do not
exist.

(Here is one proof.  In the terminology of~\cite{white-controlling}, the segment $I$ would be a
$(2,0)$ subset of $\{a<\dist(\cdot,Z)<b\}$.
But that violates various properties of $(2,0)$-subsets.
For example, a nonempty, $(2,0)$ subset of a smooth, embedded surface
  must be the whole surface~\cite{white-controlling}*{Theorem~1.3}.
In this case, we can take the surface to be $\{p: \text{$z(p)=0$ and $a<|p|<b$}\}$.)

This completes the proof that $r=R$.  By Lemma~\ref{r-R-bounds}, $r=R$ is
 positive and finite.
Thus Claim~\ref{r-R-claim} is proved.

Let 
\begin{equation}\label{Gamma-def}
  \Gamma = \{(x,z): x\ge 0, \, (x,0,z)\in \Sigma\}.
\end{equation}
Thus 
\[
  \Sigma= \{(x,y,z): (\sqrt{x^2+y^2},z)\in \Gamma\}.
\]
From Lemma~\ref{slice-lemma}, we see that $\Gamma$ consists of four curves (counting multiplicity)
that meet at the point $(r,0)$.  The arcs are geodesics with respect
to a certain metric $\tilde\gamma$ on $\{(x,z): x\ge 0\}$.
One arc, $\alpha^+$, lies in $[0,r]\times[0,\infty]$, and at every point of that arc,
\[
  -\infty \le  \frac{dz}{dx} \le  0.
\]
Another arc, $\beta^+$, lies in $[r,\infty)\times[0,\infty)$,  and 
 at every point of that arc,
\[
    0 \le  \frac{dz}{dx} \le \infty.
\]
The other arcs $\alpha^-$ and $\beta^-$ are the images of $\alpha^+$ and $\beta^+$
under $(x,z)\mapsto (x,-z)$.

\begin{claim}\label{bounded-areas-claim}
The areas of the $\Sigma_g$ are uniformly bounded (independent of $g$) in compact sets.
\end{claim}

\begin{proof}[Proof of Claim~\ref{bounded-areas-claim}]
From the description of $\Gamma$, we see that the areas of the $\Sigma_g$
are uniformly bounded on compact sets disjoint from the circle $C$.
Thus if  Claim~\ref{bounded-areas-claim} is false, then the area blowup set consists of the circle $C$.
Thus $C$ is a $(2,0)$ set (with respect to the shrinker metric).
But that violates many properties of $(2,0)$ sets. (For example, $C$ is a nonempty
subset of the connected $\gamma$-minimal surface $\{z=0\}$, so $C$ would have
to contain all of $\{z=0\}$.)

Here is an alternate proof of Claim~\ref{bounded-areas-claim}.
  If $\eps>0$ is small, then the set $J(\eps)$ of points at Euclidean
distance $\le \eps$ from $C$ has mean convex boundary,
and $J(\eps)\setminus C$ is foliated by the mean convex tori $\partial C(s)$, $0<s\le \eps$.
Hence $J(\eps)$ does not contain any closed
$\gamma$-minimal surfaces.
Thus, by~\cite{white-isoperimetric},
 $J(\eps)$ admits an isoperimetric inequality: if $S$ is a $\gamma$-minimal
surface in $J(\eps)$, then the area of $S$ is bounded by a constant $c$ times the length of $\partial S$.
Applying the isoperimetric inequality to $\Sigma_g\cap J(\eps)$ gives Claim~\ref{bounded-areas-claim}.
\end{proof}

Since the areas are uniformly bounded on compact sets, the $\Sigma_g$
converge as varifolds to $\Sigma$. 
Thus $\Sigma$ is minimal with respect to the shrinker metric.
Hence $\Gamma$ (see~\eqref{Gamma-def}) is a stationary varifold (with respect to the metric 
  $\tilde\gamma$.)
Recall that $\Gamma$ consists of four geodesic arcs, $\alpha^+$, $\beta^+$, $\alpha^-$,
and $\beta^-$, that meet at $(r,0)$.
Since $\Gamma$ is stationary, the four unit tangent vectors to those curves at $(r,0)$ sum to $0$.
It follows that the unit tangent vectors to $\alpha^+$ and $\beta^-$ at $(r,0)$ are equal and opposite.
Thus $\alpha^+$ and $\beta^-$ fit together to  form a single geodesic, and likewise for 
  $\alpha^-$ and $\beta^+$.
  
Consequently, $\Sigma$ consists of two embedded shrinkers $\Sigma'$ and $\Sigma''$ that intersect along $C$ and that are images of each other under reflection in $\{z=0\}$.

\newcommand{\Cyl}{\textnormal{Cyl}}

Now $\Sigma'$ is embedded and noncompact, has genus $0$, and contains the circle $C$.
Thus by Brendle's theorem~\cite{brendle}, $\Sigma'$ is either the plane $\{z=0\}$
or the cylinder $\Cyl:= \{\dist(\cdot,Z)=\sqrt2\}$. (Of course, in the latter case, $r$ would be $\sqrt2$.)

If $\Sigma'=\Cyl$, then $\Sigma_g\cap\{z>0\}$ converges smoothly and with multiplicity $2$ in $\{z>0\}$
to  $\Cyl^+:= \Cyl\cap\{z>0\}$, which implies 
  (cf.~\cite{white-compact}*{Theorem~1.1}) that $\Cyl^+$ is stable (for the shrinker metric).
But $\Cyl^+$ is unstable by~\cite{brendle}*{Proposition~5}.

Thus $\Sigma'$ is the plane $\{z=0\}$, so $\Sigma$ is that plane with multiplicity $2$.
This completes the proof of Theorem~\ref{flat-shrinker-theorem}.
\end{proof}

\begin{theorem}\label{flat-cone-theorem}
Suppose that $\Sigma_g\in \Ss_g$, and let $C(\Sigma_g)=\lim_{t\to 0} |t|^{1/2}\Sigma_g$.
Then the cones $C(\Sigma_g)$ converge as $g\to \infty$ to $\{z=0\}$ with multiplicity $2$,
and the convergence is smooth away from the origin.
\end{theorem}

\begin{proof}
Note that the Brakke flows
\[
t\in (-\infty,0]
\mapsto 
\begin{cases}
|t|^{1/2}\Sigma_g  &\text{if $t<0$}, \\
C(\Sigma_g)  &\text{if $t=0$},
\end{cases}
\]
converge to the constant Brakke flow 
which takes each $t< 0$ to the plane $\{z=0\}$ with multiplicity $2$.
Hence $C(\Sigma_g)$ also converges to that plane.
The smoothness of convergence follows from the fact that $\Sigma_g^+$ is graphical
in $\{\dist(\cdot,Z)\ge R(\Sigma_g)\}$ and that $R(\Sigma_g)$ converges to a finite limit as $g\to \infty$.  
\end{proof}

\begin{theorem}\label{curvature-bound-theorem}
For every $R<\infty$, there is a $c=c_R$ with the following property.
If $M$ is a shrinker (possibly with boundary) in $B(0,R)$,
if $|A(M,\cdot)|$ is bounded, and if $\nu\cdot\ee_3\ge 0$ on $M$, then
\[
   |A(M,p)| \, \dist_M(p,\partial M) \le c.
\]
\end{theorem}

Here $\dist_M$ denotes intrinsic distance in $M$. For notational simplicity, 
we will write $\dist$ for $\dist_M$ in the proof below.

\begin{proof}
Suppose not. Then there exist examples $M_i$ and $p_i\in M_i$
such that
\begin{equation}\label{blowup}
  |A(M_i,p_i)| \, \dist(p_i, \partial M_i) \to \infty.
\end{equation}
We may suppose that $p_i$ maximizes the left hand side of~\eqref{blowup}.
Translate $M_i$ by $-p_i$ and then dilate by $|A(M_i,p_i)|$ to get $M_i'$.
Then, after passing to a subsequence, the $M_i'$ converge smoothly to a complete
surface $M'$ such that is minimal with respect to the Euclidean metric,
\[
   |A(M',0)| = \max |A(M',\cdot)| = 1,
\]
and $\nu\cdot \ee_3\ge 0$ on $M'$.
Thus $M'$ is stable and therefore planar.
\end{proof}

\section{Fattening}\label{fattening-section}

In this section, we show that if $g$ is sufficiently large, then every critical
$g$-wheel fattens.

We begin with a simple lemma.

\begin{lemma}\label{cone-lemma}
Suppose that $M$ is a critical $g$-wheel and let $M(t)$, $t\in [0,T_M)$
be the associated mean curvature flow.   If $c_i\to \infty$, then, after passing to 
a subsequence, there is a $\Sigma\in \Ss_g$ such that
\[
   c_i M(T_M) \to C(\Sigma).
\]
\end{lemma}

\begin{proof}
Consider the flow
\[
t\in [-T_M,0] \mapsto M(T_M + t).
\]
Dilate parabolically by $(p,t)\mapsto (c_i p, c_i^2 t)$
to get a flow $M_i$.  After passing to a subsequence, the $M_i(\cdot)$ converge
to a flow
\[
t\in (-\infty,0]
\mapsto
\begin{cases}
|t|^{1/2}\Sigma   &(t<0), \\
C(\Sigma)    &(t=0).
\end{cases},
\]
where $\Sigma\in \Ss_g$.
\end{proof}

If $K$ is a subset of $\RR^3$, we let
\begin{equation}\label{eps-definition}
\eps(K) = \limsup_{p\in K, \, p\to 0} \frac{z(p)}{|p|}.
\end{equation}
Note that if $C\subset \RR^3$ is cone (i.e., if $C=s C$ for all $s>0$), then
\[
 \eps(C) = \max_{p\in C, \, |p|=1}\frac{z(p)}{|p|} = \max_{p\in C, \, p\ne 0} \frac{z(p)}{|p|}.
\]

\begin{corollary}\label{cone-corollary}
Suppose that $M$ is a critical $g$-wheel and let $M(t)$, $t\in [0,T_M)$
be the associated mean curvature flow.   Then there is a $\Sigma_g\in \Ss_g$
such that
\[
    \eps(C(\Sigma_g)) = \eps(M(T_M)).
\]
\end{corollary}

\begin{theorem}\label{fat-theorem}
There is an $\eps>0$ with the following property.
Suppose that $K$ is a compact subset of $\RR^3$ such that
\begin{enumerate}
\item\label{fat-1} $\eps(K)\le \eps$.
\item\label{fat-2} $K$ contains a disk $D$ in $\{z=0\}$ centered at the origin.
\item\label{fat-3} $\partial D$ is in the interior of $K$.
\end{enumerate}
Then $\partial K$ fattens immediately under level set flow.
In particular, $F_t(\partial K)$ contains $0$ in its interior for all sufficiently 
small $t>0$.
\end{theorem}

Recall that $F_t(S)$ is the result of flowing a set $S$ for time $t$ by the level set flow.

\begin{lemma}\label{fat-lemma}
Suppose that $K$ is a compact set with properties~\eqref{fat-1}, \eqref{fat-2}, and~\eqref{fat-3}
in Theorem~\ref{fat-theorem}.
Then either $\partial K$ fattens immediately,
or there is a weak set flow 
   $t\in [0,\infty)\mapsto Q(t)$ such that 
\begin{enumerate}
\item\label{fat-L1} $Q(0)\subset V_\eps:=\{p: |z(p)|\le \eps\,|p| \}$,
\item\label{fat-L2} $\dist(0,Q(t)) \ge t^{1/2}$ for $t\in [0,1]$, and
\item\label{fat-L3} $\dist(0,Q(1)) = 1$.
\end{enumerate}
\end{lemma}


\begin{proof}[Proof of Lemma~\ref{fat-lemma}]
Let 
\[
   \Aa = D \cup \{p: \dist(p,\partial D) \le \eta\},
\]
where $\eta>0$ is sufficiently small that $\Aa$ is contained in $K$.

Note that there is an $\eta>0$ such that
\[
 0\in \interior(F_t(\Aa)) \quad\text{for all $t\in (0,\eta]$},
\]
and thus
\begin{equation}\label{insider}
  0 \in \interior(F_t(K))  \quad\text{for all $t\in (0,\eta]$}.
\end{equation}

Now let $K_n = \{p\in K: \dist(p,K^c) \ge 1/n\}$.
Let $t_n$ be the first time $t>0$ such that 
\[
   \dist(0, F_t (K_n)) \le t^{1/2}.
\]
(If there is no such time, let $t_n=\infty$.)

Since $K_n \subset K_{n+1}$ for each $n$, the sequence $t_n$ is decreasing
and therefore converges to a limit $\tau\ge 0$.

{\bf Case 1}: $\tau>0$.
Then for $0<t< \min\{\tau, \eta\}$,
\[
  \text{$B(0,r(t)) \subset F_t(K)$ and $B(0,r(t))\cap F_t(K_n)=\emptyset$}
\]
for each $n$, where $r(t) = \min\{\dist(0, (F_t(K))^c),  t^{1/2}\}$, which is $>0$ 
  by~\eqref{insider}.
Thus 
\[
B(0,r(t))\subset F_t(\partial K)
\]
 for $0<t<\min\{\tau,\eta\}$, 
so (in this case) $\partial K$ fattens immediately.

{\bf Case 2}: $\tau=0$.
Let $K_n'$ be the result of dilating $K_n$ by $t_n^{-1/2}$.
Then
\begin{align*}
\dist(0, F_t(K_n')) &> t^{1/2}  \quad \text{for $t\in [0,1)$, and}   \\
\dist(0, F_1(K_n') &= 1.
\end{align*}

After passing to a subsequence, the flows $t\in [0,\infty)\mapsto F_t(K_n')$
converge to a weak set flow $t\in [0,\infty)\mapsto Q(t)$
 that has the properties asserted in the lemma.
\end{proof}

\begin{proof}[Proof of Theorem~\ref{fat-theorem}]
Suppose the theorem is false. Then, by Lemma~\ref{fat-lemma}, for each $\eps>0$,
there exists a weak set flow $t\in[0,\infty)\mapsto Q^\eps(t)$
having properties asserted in the lemma.

Let $\eps(n)>0$ be a sequence tending to $0$.
After passing to a subseqeunce, the weak
set flows $t\in [0,\infty)\mapsto Q^{\eps(n)}(t)$ converge
to a weak set flow $t\in [0,\infty) \mapsto Q'(t)$.   Note that
\[
     Q'(t) \subset \{z=0\}
\]
for $t=0$, and therefore for all $t\ge 0$.
Note also that
\begin{gather*}
\text{$\dist(0,Q'(t)) \ge t^{1/2} \quad (t\in [0,1])$, and}   \\
\dist(0,Q'(1))= 1.
\end{gather*}
But that is impossible, since a weak set flow that is a proper subset of a plane at one instant
immediately becomes empty.
\end{proof}

We can now complete the proof of Theorem~\ref{intro-theorem}:

\begin{theorem}
There is a $\hat g<\infty$ with the following property.
If $g\ge \hat g$ and if $M$ is a critical $g$-wheel,
then $M_g$ fattens at time $T_M$.
Specifically, for all $t>T_{M_g}$ sufficiently close to $T_{M_g}$, the set $F_t(M)$
contains $0$ in its interior.
\end{theorem}

\begin{proof}
Let $\eps>0$ be as in Theorem~\ref{fat-theorem}.
By Theorem~\ref{flat-cone-theorem}, there is a $\hat g$ with the following property:
if $g\ge \hat g$ and if $\Sigma\in \Ss_g$, then 
\[
   C(\Sigma) \subset V_\eps.
\]
Now suppose that $g\ge \hat g$.

By Corollary~\ref{cone-corollary}, there is a $\Sigma\in \Ss_g$ such that 
\[
    \eps(C(\Sigma)) = \eps(M(T_M)).
\]
Thus
\[
  \eps(M(T_M))  = \eps(C(\Sigma)) \le \eps(V_\eps)=\eps.
\]
Hence by Theorem~\ref{fat-theorem}, $M(T_M)$ fattens immediately.

(Note that $R:=\dist(0, \Gout(M(T_M)) > 0$, and if $D$ is a disk in $\{z=0\}$
with center $0$ and radius $<R$, then $D$ is in the compact region $K$
bounded by $M(T_M)$, and $D\setminus \{0\}$ lies in the interior of $K$.)
\end{proof}

\nocite{pedrosa-ritore}
\nocite{hoffman-wei}
\newcommand{\hide}[1]{}

\end{document}